\documentclass{amsart}

\usepackage{amsmath,amsthm,latexsym,amssymb,mathrsfs,tikz-cd}
\usepackage{mathtools}
\usepackage{stackengine,scalerel}
\usepackage[colorlinks]{hyperref}
\usepackage[capitalize]{cleveref}
\definecolor{dark-red}{rgb}{0.6,0,0}
\definecolor{dark-green}{rgb}{0,0.4,0}
\definecolor{medium-blue}{rgb}{0,0,0.5}
\hypersetup{
    colorlinks, linkcolor={dark-red},
    citecolor={dark-green}, urlcolor={medium-blue},
    pdfauthor={Matthew Bertucci and Sean Howe},
    pdftitle={Equidistribution and arithmetic \texorpdfstring{$\Lambda$}{Λ}-distributions},
}

\newcommand{\Fil}{\mr{Fil}}

\newcommand{\mc}[1]{\mathcal{#1}}
\newcommand{\mbb}[1]{\mathbb{#1}}
\newcommand{\mbf}[1]{\mathbf{#1}}
\newcommand{\mr}[1]{\mathrm{#1}}

\newcommand{\kappabar}{\overline{\kappa}}

\newcommand{\Sym}{\mathrm{Sym}}

\DeclareMathOperator{\Spec}{Spec}

\DeclareMathOperator{\Gal}{Gal}

\DeclareMathOperator{\Hom}{Hom}

\newcommand{\ul}[1]{\underline{#1}}

\newcommand{\Log}{\mathrm{Log}}
\newcommand{\res}{\mathrm{res}}

\numberwithin{equation}{subsection}
\numberwithin{equation}{subsubsection}

\newcommand{\fqbar}{\overline{\mathbb{F}}_q}
\newcommand{\fq}{\mbb{F}_q}

\theoremstyle{plain}

\newtheorem{maintheorem}{Theorem}

\newtheorem*{theorem*}{Theorem}

\newtheorem{theorem}[subsubsection]{Theorem}

\newtheorem{proposition}[subsubsection]{Proposition}
\newtheorem{lemma}[subsubsection]{Lemma}

\theoremstyle{definition}

\newtheorem{example}[subsubsection]{Example}

\newtheorem{definition}[subsubsection]{Definition}
\newtheorem{remark}[subsubsection]{Remark}

\newcommand{\Exp}{\mathrm{Exp}}

\newcommand{\lcm}{\mathrm{lcm}}
\newcommand{\bdd}{\mathrm{bdd}} 
\newcommand{\ev}{\mathrm{ev}}

\newcommand{\Fqbar}{\overline{\mathbb{F}}_q}

\newcommand{\bF}{\mathbb{F}}
\newcommand{\bN}{\mathbb{N}}
\newcommand{\bP}{\mathbb{P}}
\newcommand{\bZ}{\mathbb{Z}}
\newcommand{\bfk}{\mathbf{k}}
\newcommand{\bfone}{\mathbf{1}}
\newcommand{\frakm}{\mathfrak{m}}
\newcommand{\sheaf}[1]{\mathscr{#1}}
\newcommand{\PP}{\mathscr{P}}

\newcommand{\restr}[3][]{#2\csname#1\endcsname|\sb{#3}}

\title{Equidistribution and arithmetic $\Lambda$-distributions}

\author{Matthew Bertucci and Sean Howe}

\begin{document}

\begin{abstract}
We formulate an abstract notion of equidistribution for families of $\lambda$-probability spaces parameterized by admissible $\mathbb{Z}$-sets. Under the assumption of equidistribution, we show that the $\sigma$-moment generating functions of certain infinite sums of random variables can be computed as motivic Euler products. Combining this result with earlier generalizations of Poonen's sieve, we compute the asymptotic $\Lambda$-distributions for several natural families of function field $L$-functions and zeta functions.
\end{abstract}

\maketitle
\tableofcontents

\section{Introduction}
In \cite{Howe.RandomMatrixStatisticsAndZeroesOfLFunctionsViaProbabilityInLambdaRings}, the second author introduced a theory of probability in $\lambda$-rings in order to provide a concise language for describing random variables valued in multisets of complex numbers. The main application was a comparison between zero distributions for certain families of function field $L$-functions and associated eigenvalue distributions in random matrix statistics \cite[Theorems B and C]{Howe.RandomMatrixStatisticsAndZeroesOfLFunctionsViaProbabilityInLambdaRings}, and the key arithmetic input for the two types of families treated in \cite{Howe.RandomMatrixStatisticsAndZeroesOfLFunctionsViaProbabilityInLambdaRings} was Poonen's sieve for hypersurface sections \cite{Poonen.BertiniTheoremsOverFiniteFields}. The  purpose of the present work is to systematize, generalize, and abstract the method used in \cite{Howe.RandomMatrixStatisticsAndZeroesOfLFunctionsViaProbabilityInLambdaRings} to compute arithmetic $\Lambda$-distributions or, equivalently, their associated $\sigma$-moment generating functions. 

To that end, we first formulate in \cref{def.equidistributes} a general notion of equidistribution for a sequence of families of $\lambda$-probability spaces parameterized by an admissible $\mathbb{Z}$-set (an abstraction of the $\fqbar$-points of an algebraic variety over $\fq$). In \cref{theorem.abstract-independence} we show that, if equidistribution holds, then for certain sequences of random variables constructed by integrating over these families, the associated sequence of $\sigma$-moment generating functions converges to an explicit motivic Euler product.
We apply this abstract result to compute asymptotic $\Lambda$-distributions for $L$-function and zeta function statistics in more settings where Poonen's sieve has been generalized (using the original sieve, one recovers \cite[Theorems B and C]{Howe.RandomMatrixStatisticsAndZeroesOfLFunctionsViaProbabilityInLambdaRings}). 

In particular, in \cref{maintheorem.complete-intersections} we combine our systematization with the generalization of Poonen's sieve given in \cite{BucurKedlaya.TheProbabilityThatACompleteIntersectionIsSmooth} to compute the asymptotic $\Lambda$-distributions of the zeroes of the $L$-functions of vanishing cohomology of smooth complete intersections and compare these with the associated random matrix $\Lambda$-distributions. This generalizes the case of smooth hypersurface sections treated in \cite[Theorem C]{Howe.RandomMatrixStatisticsAndZeroesOfLFunctionsViaProbabilityInLambdaRings}. We also treat natural families arising from the semiample version of Poonen's sieve of \cite{ErmanWood.SemiampleBertiniTheoremsOverFiniteFields} and the ``smooth-agnostic" generalization of Poonen's sieve developed in  \cite{Bertucci.TaylorConditionsOverFiniteFields}. 

The key new tool that allows us to systematize and abstract the arguments of \cite{Howe.RandomMatrixStatisticsAndZeroesOfLFunctionsViaProbabilityInLambdaRings} is a motivic Euler product adapted to point-counting. We continue the introduction by explaining how this notion arises naturally in the problems we consider.  

\subsection{Equidistribution, independence, and \texorpdfstring{$\sigma$}{σ}-moment generating functions}

The arithmetic random variables studied in \cite{Howe.RandomMatrixStatisticsAndZeroesOfLFunctionsViaProbabilityInLambdaRings} are of the following form: Fix a finite field $\kappa$ and algebraic closure $\kappabar$. For each $d\geq 1$, one defines a $\lambda$-probability space where the random variables are functions on the set $U_d$ of the degree $d$ homogeneous polynomials in $n+1$ variables with coefficients in $\kappabar$ satisfying a transversality condition with respect to a fixed smooth subscheme of $\mathbb{P}^n_{\kappa}$.
For each $P \in \mathbb{P}^n(\kappabar)$, one defines a random variable $X_{d,P}$ on $U_d$ whose value on a homogeneous polynomial $F$ depends only on the Taylor expansion of $F$ at  $P$.
One then obtains a new random variable $X_d$ on $U_d$ by ``summing" the random variables $X_{d,P}$ over all points $P \in \mathbb{P}^n(\kappabar)$, and one would like to understand the asymptotic distribution of $X_d$ as $d \rightarrow \infty$ --- there is hope of this because Poonen's sieve \cite{Poonen.BertiniTheoremsOverFiniteFields} implies the Taylor expansions equidistribute in a certain natural sense.   

In this theory, the random variables are valued in $W(\mathbb{C})$, the ring of big Witt vectors of $\mathbb{C}$, and the right notion of a distribution is a $\Lambda$-distribution: When the random variable is valued in $\mathbb{Z}_{\geq 0}[\mathbb{C}] \subseteq W(\mathbb{C})$, i.e. in multisets of complex numbers, the $\Lambda$-distribution encodes the averages of all symmetric functions of the multiset, and convergence in $W(\mathbb{C})$ is simply convergence of all of these averages. 

One of the key ideas in \cite{Howe.RandomMatrixStatisticsAndZeroesOfLFunctionsViaProbabilityInLambdaRings} is to encode the $\Lambda$-distribution using the $\sigma$-moment generating function, defined for a $W(\mathbb{C})$-valued random variable $X$ as 
\[ \mathbb{E}[\Exp_{\sigma}(Xh_1)] \in \Lambda_{W(\mathbb{C})}^\wedge \]
where $\Exp_{\sigma}$ is the plethystic exponential, $h_1=t_1+t_2+t_3+\ldots$, and $\Lambda_{W(\mathbb{C})}^\wedge$ is the ring of symmetric power series with coefficients in $W(\mathbb{C})$. These behave much like the usual moment generating functions in classical probability theory: in particular, in the above setting, we expect that, as $d \rightarrow \infty$, the random variables $X_{d,P}$ will all be independent so that their moment generating functions should multiply to give
\[ \lim_{d \rightarrow \infty} \mathbb{E}[\Exp_{\sigma}(X_d h_1)] = \prod_{P \in \mathbb{P}^n(\fqbar)} \lim_{d\rightarrow \infty} \Exp_{\sigma}[(X_{d,P} h_1)]. \]
Moreover, we expect that each of the individual random variables $X_{d,P}$ should approach the distribution of a random variable $\mathcal{X}_P$ defined independently of $d$ on the space of Taylor expansions at $P$, so that this should simplify further to
\[  \prod_{P \in \mathbb{P}^n(\fqbar)} \mathbb{E}[\Exp_{\sigma}(\mathcal{X}_{P} h_1)]. \]

The main difficulty in making this heuristic precise is that it is completely unclear how one should actually define the product over $\mathbb{P}^n(\fqbar)$. In \cite{Howe.RandomMatrixStatisticsAndZeroesOfLFunctionsViaProbabilityInLambdaRings}, we made an ad hoc argument to get around this, exploiting that, in the cases treated there, the $\Lambda$-distribution of each $\mathcal{X}_P$ is the same. Under this constraint, one finds a natural candidate for the infinite product by using the pre-$\lambda$ power structure (see \cite[\S2.6]{Howe.RandomMatrixStatisticsAndZeroesOfLFunctionsViaProbabilityInLambdaRings}).

In the present work, we address the problem head-on by using the plethystic exponential and logarithm to define motivic Euler products for admissible $\mathbb{Z}$-sets  --- see \cref{def.mep}. This is motivated by joint work of the second author with Bilu and Das \cite{BiluDasHowe.SpecialValuesOfMotivicEulerProducts} where it is shown that, when working with Grothendieck rings of varieties, the same formula recovers the motivic Euler products of Bilu \cite{Bilu.MotivicEulerProductsAndMotivicHeightZetaFunctions}. Our motivic Euler products can also be treated from the perspective of Eulerian formalisms developed in \cite{BiluDasHowe.SpecialValuesOfMotivicEulerProducts}, but here we give a more direct and independent treatment. 

A key point that is specific to the case of admissible $\mathbb{Z}$-sets is \cref{prop.euler-product-point-counting-formula}, which explains how to compute the motivic Euler products in terms of classical Euler products. This formula is what allows us to build a bridge between the abstract formulation of equidistribution (\cref{def.equidistributes}), which is adapted to comparison with Poonen's sieve, and the computation of $\sigma$-moment generating functions that carries out the heuristic described above (\cref{theorem.abstract-independence}). 

\subsection{Applications}\label{ss.intro-applications}

Combining \cref{theorem.abstract-independence} with generalizations of Poonen's sieve, we can compute asymptotic distributions of many zeta function and $L$-function random variables. As an illustration, we state now a generalization of \cite[Theorem C]{Howe.RandomMatrixStatisticsAndZeroesOfLFunctionsViaProbabilityInLambdaRings} to complete intersections (that uses the generalization of Poonen's sieve in  \cite{BucurKedlaya.TheProbabilityThatACompleteIntersectionIsSmooth} to obtain the equidistribution result needed to apply \cref{theorem.abstract-independence}). Afterwards we will briefly discuss our other applications. 

\subsubsection{}
Fix a finite field $\kappa$ and an algebraic closure $\kappabar$. For $n\geq 0$, $r \geq 1$, suppose $Y \subseteq \mathbb{P}^{n}_{\kappa}$ is a smooth closed geometrically connected subscheme of dimension $m+r$. For $\ul{d}=(d_1, \ldots, d_r) \in \mathbb{N}^{r}$, let $U_{\ul{d}}$ be the set of tuples $\ul{F}=(F_1, \ldots, F_{r})$ of homogeneous polynomials of degrees $d_1, \ldots, d_r$ in $n+1$ variables with coefficients in $\kappabar$ such that the subschemes $Y, V(F_1), \ldots, V(F_r)$ are transverse at any point of intersection (i.e., the intersection of their tangent spaces at such a point is $n$-dimensional). Then the scheme-theoretic intersection $C_{\ul{F}}=Y \cap V(F_1) \cap \ldots \cap V(F_r)$ is a smooth complete intersection in $Y_{\fqbar}$ of dimension $n$. If we write $\kappa(\ul{F})$ for the subfield of $\kappabar$ generated by $\kappa$ and the coefficients of $F_1, \ldots, F_r$, then $C_{\ul{F}}$ is naturally defined over $\kappa(\ul{F})$, and as such admits a Hasse-Weil zeta function
\[ Z_{C_{\ul{F}}}(t)=\prod_{y \in |C_{\ul{F}}|}\frac{1}{1-t^{\deg{y}}}. \]
Here $|C_{\ul{F}}|$ denotes the closed points of $C_{\ul{F}}/\kappa(\ul{F})$. 
It follows from the Grothendieck-Lefschetz trace formula and the strong Lefschetz theorem in \'{e}tale cohomology that 
\[ Z_{C_{\ul{F}}}(t) =  \mathcal{L}_{{C_{\ul{F}}}}(t)^{(-1)^n} Z_0(t) \]
where $Z_0(t)$ depends only on $Y_{\kappa(\ul{F})}$ and $\mathcal{L}_{{{C_{\ul{F}}}}}(t)$ is the characteristic power series of geometric Frobenius acting on the vanishing cohomology of $C_{\ul{F}}$ (that is, the part of the cohomology which does not ``come from" $Y_{\kappa(\ul{F})}$). For $q_{\ul{F}} := \#\kappa(\ul{F})$, the reciprocal poles of $\mathcal{L}_{{{C_{\ul{F}}}}}(t)$ are $q_{\ul{F}}$-Weil numbers of weight $m$, i.e., they are algebraic integers whose conjugates all have absolute value $q_{\ul{F}}^{\frac{m}{2}}.$ One expects the renormalized characteristic series
\[ \mathcal{L}_{{{C_{\ul{F}}}}}(t q_{\ul{F}}^{\frac{-m}{2}}) \]
to behave, on average, like the characteristic power series of a random matrix in an orthogonal group if $m > 0$ is even, a compact symplectic group if $m$ is odd, or the symmetric group in its standard irreducible representation if $m=0$. 

\subsubsection{} We compute the asymptotic $\Lambda$-distribution of the random variable $X_{\ul{d}}$ on $U_{\ul{d}}$ given by
\[ X_{\ul{d}}(\ul{F})=\mathcal{L}_{{{C_{\ul{F}}}}}(t q_{\ul{F}}^{\frac{-m}{2}}) \in 1+t\mathbb{C}[[t]]=W(\mathbb{C})\]
(equivalently, sending $F$ to the multiset of reciprocal poles as an element of $\mathbb{Z}_{\geq 0}[\mathbb{C}] \subseteq W(\mathbb{C})$, i.e. to the eigenvalues of the associated matrix) and compare it to one of the classical group random matrix $\Lambda$-distributions computed in \cite[Theorem A]{Howe.RandomMatrixStatisticsAndZeroesOfLFunctionsViaProbabilityInLambdaRings}. 

To state the result, we introduce some notation: We write $[Y(\kappabar)]=Z_{Y}(t) \in W(\mbb{C})$ and $[H^i(Y)] \in 1+t\mathbb{C}[[t]] =W(\mbb{C})$ for the characteristic power series of the geometric Frobenius acting on the $i$th \'{e}tale cohomology group of $Y_{\kappabar}$ (which are relevant because the persistent factor $Z_0(t)$ of $Z_{C_{\ul{F}}}(t)$ can be expressed in terms of $[H^i(Y)]$, $i \leq m$). For $z \in \mathbb{C}$, we write $[z]$ for the element $\frac{1}{1-tz}$ in $1+t\mathbb{C}[[t]]=W(\mathbb{C})$. 

We write $e_i$ for the $i$th elementary symmetric polynomial and $h_i$ for the $i$th complete symmetric polynomial (see \cite[\S2.1]{Howe.RandomMatrixStatisticsAndZeroesOfLFunctionsViaProbabilityInLambdaRings} for the general notation and results on symmetric functions that we use). 

We set $q:=\#\kappa$. We write $L(a,b,c)=\prod_{j=0}^{c-1}(1-a^{-(b-j)})$; note that $L(q,m+r,r)$ is the probability that $r$ vectors in $\mathbb{F}_{q}^{m+r}$ are linearly independent (cf. \cite[p.2]{BucurKedlaya.TheProbabilityThatACompleteIntersectionIsSmooth}). 

We write 
$\displaystyle \lim_{\ul{d} \xrightarrow{*} \infty}$
for a limit where each $d_i \rightarrow \infty$ and $\max(d_i)^{m+r} q^{\frac{-\min(d_i)}{m+r+1}} \rightarrow 0$. 

Finally, we write $W(\mathbb{C})^\bdd$ for the subring of $W(\mathbb{C})$ consisting of elements with bounded ghost components (which, as in \cite[\S9]{Howe.RandomMatrixStatisticsAndZeroesOfLFunctionsViaProbabilityInLambdaRings}, encodes some big $O$ notation).  

When $r=1$, the following result is \cite[Theorem C]{Howe.RandomMatrixStatisticsAndZeroesOfLFunctionsViaProbabilityInLambdaRings}; we refer the reader to loc. cit. and the surrounding discussion for more on its classical interpretations. 

\begin{maintheorem}\label{maintheorem.complete-intersections}
    With notation as above, 
\begin{equation}\label{eq.ci-main-theorem-limit} \lim_{\ul{d} \xrightarrow{*} \infty}\mbb{E}[\Exp_{\sigma}(X_dh_1)]= (1 + p\sum_{i \geq 1} [q^{-im/2}] \epsilon^i f_i)^{[Y(\kappabar)]} \cdot \Exp_\sigma(\mu h_1)  \end{equation}
where $f_i=e_i$ and $\epsilon=-1$ if $m$ is odd and $f_i=h_i$ and $\epsilon=1$ if $m$ is even, 
\begin{multline*} p = \frac{[q]^{-r}L([q],m+r,r)}{1- [q]^{-r} + [q]^{-r}L([q],m+r,r)} \textrm{, and } \\\mu=-\epsilon \left(\sum_{i=0}^{m-1} (-1)^i\left([q^{\frac{-m}{2}}]+[q^{\frac{m-2i}{2}}]\right)[H^i(Y)]\right) - [q^{-m/2}][H^{m}(Y)]. \end{multline*}
In particular, modulo $[q^{-\frac{1}{2}}]\Lambda_{W(\mbb{C})^\bdd}^\wedge$ this agrees with:
    \begin{enumerate}
        \item For $n>0$ even, $\Exp_{\sigma}(h_2)$ (the asymptotic $\sigma$-moment generating function for orthogonal random matrices  \cite[Theorem A-(1)]{Howe.RandomMatrixStatisticsAndZeroesOfLFunctionsViaProbabilityInLambdaRings}).
        \item For $n$ odd, $\Exp_{\sigma}(e_2)$ (the asymptotic $\sigma$-moment generating function for symplectic random matrices \cite[Theorem A-(2)]{Howe.RandomMatrixStatisticsAndZeroesOfLFunctionsViaProbabilityInLambdaRings}).
        \item For $n=0$, $\Exp_{\sigma}(h_2+h_3+\ldots)$ (the asymptotic $\sigma$-moment generating function for the standard irreducible representations of symmetric groups \cite[Example 1.2.1]{Howe.RandomMatrixStatisticsAndZeroesOfLFunctionsViaProbabilityInLambdaRings}).  
    \end{enumerate}
\end{maintheorem}

To obtain the limiting $\sigma$-moment generating function in \cref{maintheorem.complete-intersections}, as in the proof of \cite[Theorem C]{Howe.RandomMatrixStatisticsAndZeroesOfLFunctionsViaProbabilityInLambdaRings}, we use formal properties of independence to reduce to a geometric version computing the distribution of the random variable $Z_{\ul{d}}$ sending $\ul{F}$ to $Z_{C_{\ul{F}}}(t)$. This geometric result is given in \cref{theorem.ci-geometric-rv}, and generalizes \cite[Theorem 8.3.1]{Howe.RandomMatrixStatisticsAndZeroesOfLFunctionsViaProbabilityInLambdaRings} (we use \cite[Theorem 2.2.1]{Howe.TheNegativeSigmaMomentGeneratingFunction} to handle a negative sign when $m$ is odd instead of carrying out the computation separately as in \cite{Howe.RandomMatrixStatisticsAndZeroesOfLFunctionsViaProbabilityInLambdaRings}). The comparisons with random matrix statistics are then deduced as in \cite[Proposition 9.2.2]{Howe.RandomMatrixStatisticsAndZeroesOfLFunctionsViaProbabilityInLambdaRings}  

\subsubsection{}
In \cref{theorem.exotic-tranversality-zeta} we give a different generalization of \cite[Theorem 8.3.1]{Howe.RandomMatrixStatisticsAndZeroesOfLFunctionsViaProbabilityInLambdaRings}, treating zeta functions of hypersurface sections with more exotic transversality conditions --- this is deduced from \cref{theorem.abstract-independence} using the smooth-agnostic extension of Poonen's sieve in \cite{Bertucci.TaylorConditionsOverFiniteFields}. In \cref{ss.example-L-functions-of-characters} we also explain how the computations \cite[Theorem B]{Howe.RandomMatrixStatisticsAndZeroesOfLFunctionsViaProbabilityInLambdaRings}, which gave the asymptotic $\Lambda$-distributions of $L$-functions of certain families of Dirichlet characters, can be recovered from the perspective adopted here (this requires only Poonen's original sieve as input into \cref{theorem.abstract-independence})

In \cref{theorem.hirzebruch-2d} we compute the asymptotic $\Lambda$-distributions of the zeta functions of degree $(2,d)$ curves on Hirzebruch surfaces  --- this computation is deduced from \cref{theorem.abstract-independence} using the semiample extension of Poonen's sieve in \cite{ErmanWood.SemiampleBertiniTheoremsOverFiniteFields}, and it refines \cite[Theorem 9.9-(b)]{ErmanWood.SemiampleBertiniTheoremsOverFiniteFields} (see the start of \cref{ss.hirzebruch} for further discussion). 

\begin{remark}
    Analogous results in the Grothendieck rings of varieties and Hodge structures will be given in \cite{BiluHowe.MotivicRandomVariables}. In that setting, the results and methods of \cite{BiluHowe.MotivicEulerProductsInMotivicStatistics} provide a uniform treatment of the input that is analogous to that coming from Poonen's sieve in the point-counting setting. We recall that results giving asymptotic point-counts or traces of Frobenius and results giving asymptotics in the Grothendieck ring of varieties with respect to the dimension filtration do not imply one another, and typically require different methods  (see, e.g., \cite[\S1]{BiluDasEtAl.ZetaStatisticsAndHadamardFunctions} for a detailed discussion).  
\end{remark}

\subsection{Organization} In \cref{s.preliminaries} we set up some basic results on (pre-)$\lambda$-rings, admissible $\mathbb{Z}$-sets and their $W(\mathbb{C})$-valued functions, and $\lambda$-probability spaces. Much of the material is recalled from \cite{Howe.RandomMatrixStatisticsAndZeroesOfLFunctionsViaProbabilityInLambdaRings}, but we also give a few new results and definitions adapted to handling families of $\lambda$-probability spaces parameterized by admissible $\mathbb{Z}$-sets. In \cref{s.motivic-euler-products} we define motivic Euler products with respect to a map of admissible $\mathbb{Z}$-sets and establish the basic properties of this operation. 

After these preliminaries, in \cref{s.equidistribution-and-independence} we define our notion of equidistribution and establish the abstract form of our main result, \cref{theorem.abstract-independence}. 
In \cref{s.homog}, \cref{s.tuples-homog}, and \cref{s.semiample}, respectively, we then use generalizations of Poonen's sieve to establish equidistribution and obtain applications of \cref{theorem.abstract-independence} for homogeneous polynomials (using the generalization of Poonen's sieve of \cite{Bertucci.TaylorConditionsOverFiniteFields}), tuples of homogeneous polynomials (using the generalization of Poonen's sieve in \cite{BucurKedlaya.TheProbabilityThatACompleteIntersectionIsSmooth}), and sections of semiample line bundles (using the generalization of Poonen's sieve in \cite{ErmanWood.SemiampleBertiniTheoremsOverFiniteFields}), respectively. In particular, \cref{maintheorem.complete-intersections} is established in \cref{s.tuples-homog}. 

\subsection{Acknowledgments}
Matthew Bertucci was partially supported during the preparation of this work by the University of Utah's NSF Research Training Grant \href{https://www.nsf.gov/awardsearch/showAward?AWD_ID=1840190}{\#1840190}. Sean Howe was supported by the National Science Foundation through grant \href{https://www.nsf.gov/awardsearch/showAward?AWD_ID=2201112}{DMS-2201112}. We thank Margaret Bilu and Ronno Das for helpful conversations, especially related to the motivic Euler product formula of \cite{BiluDasHowe.SpecialValuesOfMotivicEulerProducts}.

\section{Preliminaries}\label{s.preliminaries}

In this section we discuss some basic notions on (pre-)$\lambda$-rings, admissible $\mathbb{Z}$-sets and their $W(\mathbb{C})$-valued functions, and $\lambda$-probability. Except for some new base change properties that will be useful for working with geometric families of random variables, this material is treated in more detail in \cite[\S2, 3, and 5]{Howe.RandomMatrixStatisticsAndZeroesOfLFunctionsViaProbabilityInLambdaRings}. We give citations to \cite{Howe.RandomMatrixStatisticsAndZeroesOfLFunctionsViaProbabilityInLambdaRings}; citations to earlier work on some of these topics can be found in \cite{Howe.RandomMatrixStatisticsAndZeroesOfLFunctionsViaProbabilityInLambdaRings}. 

\subsection{pre-\texorpdfstring{$\lambda$}{λ}-rings}\label{ss.pre-lambda-prelim}

\subsubsection{}We will use the notation for symmetric functions described in \cite[\S2.1]{Howe.RandomMatrixStatisticsAndZeroesOfLFunctionsViaProbabilityInLambdaRings}. In particular, we write $\Lambda$ for the ring of symmetric functions, and $h_i$ (resp. $e_i$, resp. $p_i$) denotes the $i$th complete (resp. elementary, resp. power sum) symmetric function. For $\tau=(\tau_1,\tau_2,\ldots)$ a partition, $h_\tau=h_{\tau_1}h_{\tau_2}\cdots$, and similarly for $p_\tau$ and $e_\tau$. The monomial symmetric function $m_\tau$ is the formal sum of all distinct permutations of the monomial $t_1^{\tau_1}t_2^{\tau_2}\cdots$. We will frequently use that
\[ h_1=e_1=p_1=m_{(1,0,0,\ldots)}=t_1+t_2+t_3+\ldots.\]

\subsubsection{} Recall from \cite[\S2.2 and \S2.4]{Howe.RandomMatrixStatisticsAndZeroesOfLFunctionsViaProbabilityInLambdaRings} that a (pre-)$\lambda$-ring $R$ is a ring equipped with a plethystic action of the ring $\Lambda$ of symmetric functions, written $a \circ r$ for $a \in \Lambda$ and $r \in R$, satisfying certain natural compatibilities.

\subsubsection{}\label{sss.witt}
We let $W(\mathbb{C})=\Hom_{\mathrm{ring}}(\Lambda, \mathbb{C})$ denote the ring of big Witt vectors of $\mathbb{C}$. We refer the reader to \cite[\S 5.1]{Howe.RandomMatrixStatisticsAndZeroesOfLFunctionsViaProbabilityInLambdaRings} for an overview of its properties; here we briefly recall the structures we will use.  As an additive group, $W(\mathbb{C})$ is naturally isomorphic to $1+t\mathbb{C}[[t]]$ under multiplication, and for $w \in W(\mathbb{C})$, we write $w(t^i)$ for the element obtained by substituting $t$ for $t^i$ in this presentation.  As a ring $W(\mathbb{C})$ is naturally isomorphic to $\prod_{k \geq 1} \mathbb{C}$. For $w \in W(\mathbb{C})$, and $k \geq 1$ we write $w_k$ for its $k$th component in this product presentation, called the $k$th ghost component. There is a $\lambda$-ring structure on $W(\mathbb{C})$ determined by the Adams operations 
\begin{equation}\label{eq.wv-adams-operations}
p_i \circ (w_1, w_2, \ldots)=(w_i, w_{2i}, \ldots)
\end{equation}
and we have 
\begin{equation}\label{eq.wv-substitution-ghost}  w(t^i)_k=i w_{i/k} 
\end{equation}
where  we take $w_{i/k}$ to be zero if $i/k$ is not a positive integer. 

\subsubsection{}
Recall from \cite[Section 2.3]{Howe.RandomMatrixStatisticsAndZeroesOfLFunctionsViaProbabilityInLambdaRings} that, for any (pre-)$\lambda$-ring $R$, we have a natural (pre-)$\lambda$-ring structure on 
\[ R[[\ul{t}_{\mathbb{N}}]] =\varprojlim_{n} R[[t_1, \ldots, t_n]] \]
extending the (pre-)$\lambda$-ring structure on $R$ and such that $p_i \circ t_j = t_j^i$. It is moreover a filtered (pre-)$\lambda$-ring for the filtration by monomial degree. In particular, as in \cite[\S2.5]{Howe.RandomMatrixStatisticsAndZeroesOfLFunctionsViaProbabilityInLambdaRings}, we have the $\sigma$-exponential 
\[ \Exp_{\sigma}: \Fil^1 R[[\ul{t}_{\mathbb{N}}]] \xrightarrow{\sim} 1+ \Fil^1 R[[\ul{t}_{\mathbb{N}}]],\; F \mapsto \sum_{k \geq 0} h_k \circ F \]
and its inverse (which exists for formal reasons)
\[ \Log_{\sigma}: 1+ \Fil^1 R[[\ul{t}_{\mathbb{N}}]] \xrightarrow{\sim} \Fil^1 R[[\ul{t}_{\mathbb{N}}]]. \] 
All of these constructions can be restricted to the (pre-)$\lambda$-subring $\Lambda_R^{\wedge} \subseteq R[[\ul{t}_{\mbb{N}}]]$ of symmetric power series.

\subsubsection{}\label{sss.powers}Recall that, for $F \in 1 + \Fil^1R[[\ul{t}_{\mathbb{N}}]]$ and $N \in R[[\ul{t}_{\mathbb{N}}]]$ we have an associated pre-$\lambda$ power as in \cite[Definition 2.6.1]{Howe.RandomMatrixStatisticsAndZeroesOfLFunctionsViaProbabilityInLambdaRings}, 
\[F^N := \Exp_{\sigma}(N \cdot\Log_\sigma(F)).\]

\subsubsection{}\label{sss.coeff-wise}At certain points, we will also wish to use the coefficient-wise pre-$\lambda$-ring structure on $R[[\ul{t}_{\mathbb{N}}]]$ or $\Lambda_R^\wedge$, which we denote by $\ast$, i.e. $f \ast \sum r_{\ul{i}} \ul{t}^{\ul{i}}= \sum (f \circ r_{\ul{i}}) \ul{t}^{\ul{i}}.$

\subsection{Admissible \texorpdfstring{$\mathbb{Z}$}{ℤ}-sets}\label{ss.admissible-z-prelim}

\subsubsection{}Recall from \cite[\S5.2]{Howe.RandomMatrixStatisticsAndZeroesOfLFunctionsViaProbabilityInLambdaRings} that an admissible $\mathbb{Z}$-set is a set $V$ with an action of $\mathbb{Z}$ such that $V=\cup_{k \geq 1} V^{k\mathbb{Z}}$ and, for any $k \geq 1$, $V^{k\mathbb{Z}}$ is finite. 
We write $|V|$ for the set of $\mathbb{Z}$-orbits in $V$ and, for $v \in V$, we write $|v|$ for the orbit containing $v$. The degree of a point or orbit is the size of the orbit; we write this as $\deg(v)$ or $\deg(|v|)$. We write $\mathbf{k}=\mathbb{Z}/k\mathbb{Z}$ (as an admissible $\mathbb{Z}$-set). For any admissible $\mathbb{Z}$-set $V$, we write 
\[ [V] = \prod_{|v| \in |V|} \frac{1}{1-t^{\deg(|v|)}} \in 1+t\mathbb{C}[[t]]= W(\mathbb{C}). \]
In ghost coordinates, we have $[V]=(\#V(\mathbf{1}), 
\#V(\mathbf{2}), \ldots)$ (see \cite[Lemma 5.2.3]{Howe.RandomMatrixStatisticsAndZeroesOfLFunctionsViaProbabilityInLambdaRings}) 

\subsubsection{}\label{sss.fiber}Recall from \cite[Definition 5.2.5]{Howe.RandomMatrixStatisticsAndZeroesOfLFunctionsViaProbabilityInLambdaRings} that, if $V \rightarrow B$ is a map of admissible $\mathbb{Z}$-sets, then, for any $b \in B$, we equip the fiber $V_b$ with the structure of an admissible $\mathbb{Z}$-set by multiplying the action on $V$ by $\deg(b)$ so that it restricts to an action on $V_b$. 

\subsubsection{}\label{sss.ext-of-k}For $V$ an admissible $\mathbb{Z}$-set and $k$ a positive integer, we write $V_{\mbf{k}}$ for the same set but with the action of $\mathbb{Z}$ multiplied by $k$. If $\varphi: V \rightarrow B$ is a map of admissible $\mathbb{Z}$-sets then we write $\varphi_{\mbf{k}}$ for the induced map $V_{\mbf{k}} \rightarrow B_{\mbf{k}}$. 

We can identify $V_{\mbf{k}}$ with the fiber of the projection map $V \times \mbf{k} \rightarrow \mbf{k}$ over $1 \in \mbf{k}$.

\subsubsection{}
Admissible $\mathbb{Z}$-sets give an abstract formalism for studying the points of algebraic varieties over finite fields: for $\mathbb{F}_q$ a finite field of cardinality $q$, fix an algebraic closure $\fqbar$ and write $\mathbb{F}_{q^k}$ for the unique subfield of cardinality $q^k$. Then, the set of $\overline{\mathbb{F}}_q$-points $Y(\overline{\mathbb{F}}_q)$ of an algebraic variety over $Y/\mathbb{F}_q$ is an admissible $\mathbb{Z}$-set with $1$ acting as the geometric Frobenius; in this case, $|Y(\overline{\mathbb{F}}_q)|$ can be identified with the closed points of $Y$, and $[Y(\fqbar)]$ is an incarnation of the zeta function of $Y$ (see also \cite[Example 5.2.2]{Howe.RandomMatrixStatisticsAndZeroesOfLFunctionsViaProbabilityInLambdaRings}). The set $Y(\Fqbar)_{\mbf{k}}$ is then $Y_{\mathbb{F}_{q^k}}(\fqbar)$, i.e. the set $Y(\fqbar)$ but with $1 \in \mathbb{Z}$ acting by the $q^k$-power geometric Frobenius instead of the $q$-power geometric Frobenius.

\subsection{\texorpdfstring{$W(\mathbb{C})$}{W(ℂ)}-valued functions}
\subsubsection{} For $V$ an admissible $\mathbb{Z}$-set, we write $C(V, W(\mathbb{C}))$ for the set of functions from $V$ to $W(\mathbb{C})$ that are constant on $\mathbb{Z}$-orbits. It is a $\lambda$-ring with the pointwise $\lambda$-ring structure, and, given $\varphi: V \rightarrow B$ a  map of admissible $\mathbb{Z}$-sets, there is a natural pullback map (\cite[Definition 5.3.4-(1)]{Howe.RandomMatrixStatisticsAndZeroesOfLFunctionsViaProbabilityInLambdaRings}) 
\[ \varphi^*: C(B, W(\mathbb{C})) \rightarrow C(V, W(\mathbb{C})) \]
that is a map of $\lambda$-rings, and a natural integration-over-fibers map (\cite[Definition 5.3.4-(2)]{Howe.RandomMatrixStatisticsAndZeroesOfLFunctionsViaProbabilityInLambdaRings})
\[ \varphi_!: C(V, W(\mathbb{C})) \rightarrow C(B,W(\mathbb{C}))\]  that is $\varphi^*$-linear (\cite[Lemma 5.3.7]{Howe.RandomMatrixStatisticsAndZeroesOfLFunctionsViaProbabilityInLambdaRings}). When the map $\varphi$ is clear from context, we  write $\int_{V/B}$ in place of $\varphi_!$. If $B=\mathbf{1}$ is the final object, we may also write $\int_V$ in place of $\int_{V/\mathbf{1}}$. We recall the formulas for $\varphi^*$ and $\int_{V/B}$, since they will be used below:
\begin{itemize}
    \item For $g \in C(B,W(\mathbb{C}))$, 
\[ (\varphi^* g)(v)=p_{\frac{\deg(v)}{\deg(\varphi(v))}}\circ g(\varphi(v)).\]
    \item For $f \in C(V, W(\mathbb{C}))$,
    \[ \left(\int_{V/B} f\right)(b)=\int_{V_b} f = \sum_{|v| \in |V_b|} f(t^{\deg(|v|)}). \]
\end{itemize}

\begin{example}[See Example 5.3.5 of \cite{Howe.RandomMatrixStatisticsAndZeroesOfLFunctionsViaProbabilityInLambdaRings}] \label{example.integral-of-1-is-zeta}
 For $V$ an admissible $\mathbb{Z}$-set, \[  \int_{V} 1 = [V] \]
where here $1$ denotes the constant function on $V$ with value the unit in $W(\mathbb{C})$ (the unit in $W(\mathbb{C})$ is the element $\frac{1}{1-t}$ in the identification $1+t\mathbb{C}[[t]]=W(\mathbb{C})$). 
\end{example}

We will need the following base change formula relating pullback and integration over fibers that was not made explicit in \cite{Howe.RandomMatrixStatisticsAndZeroesOfLFunctionsViaProbabilityInLambdaRings}: 

\begin{lemma}\label{lemma.base-change}
If 
\[\begin{tikzcd}
	{V_1} & {V_2} \\
	{B_1} & {B_2}
	\arrow["\varphi", from=1-1, to=1-2]
	\arrow[from=1-1, to=2-1]
	\arrow[from=1-2, to=2-2]
	\arrow["\psi", from=2-1, to=2-2]
\end{tikzcd}\]
is a cartesian diagram of admissible $\mathbb{Z}$-sets then, for $f \in C(V_2, W(\mathbb{C}))$, 
\[ \psi^* \int_{V_2/B_2} f = \int_{V_1/ B_1} \varphi^*f. \]
\end{lemma}
\begin{proof}
We write $V:=V_2$ so that $V_1=V_2 \times_{B_2} B_1$. Then, for $b \in B_1$ and $a:=\deg(b)/\deg(\psi(b))$, 
\begin{align} \nonumber \left( \psi^* \int_{V/B_2} f \right) (b) &= p_{a}\circ \left(\int_{V/B_2}f\right)(\psi(b)) \\
\nonumber &= p_{a}\circ \int_{V_{\psi(b)}}f|_{V_{\psi(b)}} \\
\label{eq.first-expansion-base-change}&= \sum_{|v| \in V_{\psi(b)}} p_{a}\circ \left(f(|v|)(t^{\deg(|v|)})\right), 
\end{align}
where the three equalities are immediate from the definition of pullback and integration. On the other hand, again immediately from the definitions, 
\begin{equation}\label{eq.second-expansion-base-change} \left( \int_{V \times_{B_2} B_1 / B_1} \varphi^*f \right) (b) = \sum_{|w| \in V_b} \left(p_{\deg(w)/\deg(\varphi(w))}\circ f(|w|)\right)(t^{\deg(|w|)}). \end{equation}
Now, $V_{b}$ is identified with $V_{\psi(b)}$, but with the action multiplied by $a$. Thus each orbit $|v|$ in $V_{\psi(b)}$, viewed as a subset of $V_b$, decomposes into $\mu_{|v|}:=\gcd(\deg(|v|), a)$ orbits of size $\deg(|v|)/\mu_{|v|}$, whose points $w$ satisfy $\deg(w)/\deg(v)=a/\mu_{|v|}$.  Thus, we can rewrite the sum on the right of \cref{eq.second-expansion-base-change} as 
\begin{equation}\label{eq.second-expansion-base-change-v2} \sum_{|v| \in V_{\psi(b)}} \mu_{|v|}\left( p_{a/\mu_{|v|}}\circ f(|v|)\right)(t^{\deg(|v|)/\mu_{|v|})}). \end{equation} 

To see that \cref{eq.second-expansion-base-change-v2} agrees with the right of \cref{eq.first-expansion-base-change}, we use \cref{eq.wv-adams-operations} and \cref{eq.wv-substitution-ghost} to compute ghost coordinates of the terms appearing as
\[  \left( p_{a}\circ \left(f(|v|)(t^{\deg(|v|)})\right) \right)_i = \deg(|v|) f(|v|)_{ai/\deg(|v|)} \]
and 
\[ \left(\left( p_{a/\mu_{|v|}}\circ f(|v|)\right)(t^{\deg(|v|)/\mu_{|v|})})\right)_i = \frac{\deg(|v|)}{\mu_{|v|}} f(|v|)_{ai/\deg(|v|)}. \]
Since a multiple of $\mu_{|v|}$ appears in \cref{eq.second-expansion-base-change-v2}, we conclude. 
\end{proof}

\subsubsection{}\label{sss.restriction-to-Vk} Recall from \cref{sss.ext-of-k} that, for $V$ an admissible $\mathbb{Z}$-set and $k$ a positive integer, we have defined $V_{\mathbf{k}}$ by multiplying the $\mathbb{Z}$-action by $k$. 

If $f \in C(V, W(\mathbb{C}))$ and $k \geq 1$ then from $f$ we obtain $f_{V_{\mbf{k}}} \in C(V_{\mbf{k}}, W(\mathbb{C}))$ by identifying $V_{\mbf{k}}=(V \times \mathbf{k})_1$ as in \cref{sss.ext-of-k}, and pulling back first to $V \times \mathbf{k}$ then restricting. Concretely, if $x$ is a point of degree $d$ in $V$ then it is a point of degree $d/\gcd(d,k)$ in $V_{\mbf{k}}$, and $f_{V_{\mbf{k}}}(x)=p_{k/\gcd(d,k)} \circ f(x)$. We will often write $f$ in place of $f_{V_{\mbf{k}}}$ when the domain is clear; we emphasize that this is not the same as transporting $f$ naively along the identification off the underlying sets of $V$ and $V_k$.

\subsection{Some \texorpdfstring{$\lambda$}{λ}-probability spaces}
\subsubsection{} We recall from \cite[Definition 3.1.1]{Howe.RandomMatrixStatisticsAndZeroesOfLFunctionsViaProbabilityInLambdaRings} that a {(pre-)}$\lambda$-probability space is a {(pre-)}$\lambda$-ring $R$ equipped with a $\mathbb{Z}$-linear expectation functional $\mathbb{E}: R \rightarrow C$ to another ring $C$ such that $\mathbb{E}[1_R]=1_C$. 

Given a (pre-)$\lambda$-probability space $(R, \mathbb{E})$, we refer to the elements of $R$ as random variables. Given a random variable $X \in R$, we recall from \cite[Lemma 3.2.2]{Howe.RandomMatrixStatisticsAndZeroesOfLFunctionsViaProbabilityInLambdaRings} that the $\Lambda$-distribution of $X$ as in \cite[Definition 3.1.2]{Howe.RandomMatrixStatisticsAndZeroesOfLFunctionsViaProbabilityInLambdaRings} is determined by the $\sigma$-moment generating function of $X$ \cite[Definition 3.2.1]{Howe.RandomMatrixStatisticsAndZeroesOfLFunctionsViaProbabilityInLambdaRings}. This latter is defined as
\[ \mathbb{E}[\Exp_{\sigma}(Xh_1)] \in \Lambda_C^\wedge \]
where $h_1=p_1=e_1=t_1+t_2+t_3 + \ldots$ is the first complete, power sum, and elementary symmetric function, $\Exp_{\sigma}(Xh_1)$ is computed in $\Lambda_R^\wedge$ (or equivalently in $R[[\ul{t}_{\mathbb{N}}]]$) and the expectation is applied coefficient-wise to produce an element of $\Lambda_C^\wedge$. 

\subsubsection{} Suppose $V$ is an admissible $\mathbb{Z}$-set and $V(\mathbf{1})\neq\emptyset$. Then, $V(\mathbf{k}) \neq \emptyset$ for all $k \geq 1$, so 
$[V]=(\#V(\mathbf{1}), \#V(\mathbf{2}), \ldots)$ is an invertible element of $W(\mathbb{C})$. As in \cite[\S5.5]{Howe.RandomMatrixStatisticsAndZeroesOfLFunctionsViaProbabilityInLambdaRings}, we can thus consider the $\lambda$-probability space
\[ (C(V, W(\mathbb{C})), \mathbb{E} ) \]
where the expectation functional $\mathbb{E}$ is defined by 
\[ \mathbb{E}:C(V, W(\mathbb{C})) \rightarrow W(\mathbb{C}), f \mapsto \frac{\int_V f}{[V]}. \]

\subsubsection{} Given a pre-$\lambda$-probability space with expectation $\mathbb{E}$ valued in $W(\mathbb{C})$, we write $\mbb{E}_k$ for the $\mathbb{C}$-valued expectation obtained by projecting to the $k$th ghost component. We recall that, for the $\lambda$-probability space $(C(V, W(\mathbb{C})), \mathbb{E})$ as above, $\mathbb{E}_k$ can be computed naturally on a classical finite probability space \cite[Lemma 5.5.1]{Howe.RandomMatrixStatisticsAndZeroesOfLFunctionsViaProbabilityInLambdaRings}. We now give a modified formulation that is more convenient for our purposes.  

To state it, note that we can view $V_{\mbf{k}}(\bf{1})$ as the subset of $\mathbb{Z}$-fixed points in $V_{\mbf{k}}$ 

\begin{lemma}\label{lemma.restriction-map}
Let $V$ be an admissible $\mathbb{Z}$-set with $V(\bf{1})\neq \emptyset$. Restriction as in \cref{sss.restriction-to-Vk} from $V$ to $V_{\mbf{k}}$ followed by naive restriction from $V_{\mbf{k}}$ to $V_{\mbf{k}}(\bf{1})$ induces a map of $\lambda$-probability spaces
    \[ \res_k: (C(V, W(\mathbb{C})), \mathbb{E}_k) \rightarrow (C( V_{\mbf{k}}(\mbf{1}), W(\mathbb{C})),  \mathbb{E}_1). \]
    In particular, for any random variable $X \in C(V, W(\mathbb{C}))$, 
    \[ \mathbb{E}_k\left[\Exp_{\sigma}(Xh_1)\right] = \mathbb{E}_1\left[\Exp_{\sigma}(\res_k(X)h_1)\right].\]
\end{lemma}
\begin{proof}
Noting that evaluation at $1 \in \mbf{k}$ gives a canonical identification $V(\mbf{k})=V_{\mbf{k}}(\mbf{1})$, this is a reformulation of \cite[Lemma 5.5.1]{Howe.RandomMatrixStatisticsAndZeroesOfLFunctionsViaProbabilityInLambdaRings}.     
\end{proof}

\begin{remark}
The map $\res_k$ is a map of $W(\mathbb{C})$-algebras when $C(V,W(\mathbb{C}))$ is equipped with the  algebra structure of pullback from a point and $C(V_{\mbf{k}}(\mbf{1}), W(\mathbb{C}))$ is equipped with the algebra structure sending $a \in W(\mathbb{C})$ to the constant function with value $p_k \circ a$. In the statement of  \cite[Lemma 5.5.1]{Howe.RandomMatrixStatisticsAndZeroesOfLFunctionsViaProbabilityInLambdaRings}, this twist in the algebra structure is included in the notation by writing  $C(V(\bfk),W(\mathbb{C}))^{(k)}$ in place of $C(V(\bfk), W(\mathbb{C}))$. 
\end{remark}

\subsubsection{}\label{sss.families}We will also need to consider ``families" of $\lambda$-probability spaces: if $V \rightarrow B$ is a morphism of admissible $\mathbb{Z}$-sets admitting a section $B \rightarrow V$, then the class $[V/B] \in C(B, W(\mathbb{C}))$  of \cite[Example 5.3.5]{Howe.RandomMatrixStatisticsAndZeroesOfLFunctionsViaProbabilityInLambdaRings} sending $b \in V$ to $[V_b]$ (for $V_b$ as in \cref{sss.fiber}) is invertible and we can consider the $\lambda$-probability space
\[ (C(V, W(\mathbb{C})), \mathbb{E}_{V/B} ) \]
where the expectation functional $\mathbb{E}_{V/B}$ is defined by 
\[ \mathbb{E}_{V/B}:C(V, W(\mathbb{C})) \rightarrow C(B, W(\mathbb{C})),\; f \mapsto \frac{\int_{V/B} f}{[V/B]}. \]

\begin{lemma}\label{lemma.expectation-in-families-values}
    For $f \in C(V,W(\mathbb{C}))$, 
    \[ \left(\mathbb{E}_{V/B}[f]\right)(b)=\mathbb{E}_{V_b}(f|_{V_b}). \] 
\end{lemma}
\begin{proof}
This follows from the analogous properties in the definitions of $\int_{V/B}$ (\cite[Definition 5.3.4]{Howe.RandomMatrixStatisticsAndZeroesOfLFunctionsViaProbabilityInLambdaRings}) and $[V/B]$ (\cite[Example 5.3.5]{Howe.RandomMatrixStatisticsAndZeroesOfLFunctionsViaProbabilityInLambdaRings}): 
    \[ \left(\mathbb{E}_{V/B}[f]\right)(b) = \frac{\left(\int_{V/B} f\right) (b)}{[V/B](b)}= \frac{\int_{V_b} f|_{V_b}} {[V_b]}=\mathbb{E}_{V_b}[f|_{V_b}]. \]
\end{proof}

\begin{lemma}\label{lemma.pullback-expectations} Consider a cartesian diagram of admissible $\mathbb{Z}$-sets
\[\begin{tikzcd}
	{V_1} & {V_2} \\
	{B_1} & {B_2}
	\arrow["\varphi", from=1-1, to=1-2]
	\arrow[from=1-1, to=2-1]
	\arrow[from=1-2, to=2-2]
	\arrow["\psi", from=2-1, to=2-2]
\end{tikzcd}\]
and suppose $V_2 \rightarrow B_2$ admits a section. Then, for $f \in C(V_2, W(\mathbb{C}))$, 
\[ \psi^* \mathbb{E}[f] = \mathbb{E}[\varphi^* f]. \]
\end{lemma}
\begin{proof}
We have 
\[ \psi^*\mathbb{E}_{V_2/B_2}[f]=\frac{\psi^* \int_{V_2/B_2} f}{\psi^*[V_2/B_2]}=\frac{\int_{V_1/B_1}\varphi^*f}{[V_1/B_1]}=\mathbb{E}_{V_1/B_1}[\varphi^*f],\]
where the second equality is by \cref{lemma.base-change} (note $[V/B]=\int_{V/B} 1$). 
\end{proof}

\begin{lemma}\label{lemma.expectations-change-of-k}
   For $f \in C(V,W(\mathbb{C}))$, 
    \[ \left(\mathbb{E}_{V/B}[f]\right)_{B_{\mbf{k}}}=\mathbb{E}_{V_{\mbf{k}}/B_{\mbf{k}}}[f_{V_{\mbf{k}}}]. \] 
\end{lemma}
\begin{proof}
Identifying $B_{\mbf{k}}=(B \times \mbf{k})_1$ as in \cref{sss.ext-of-k} and recalling the definition in \cref{sss.restriction-to-Vk}, we apply \cref{lemma.pullback-expectations} to see
\[ \left(\mathbb{E}_{V/B}[f] \right)_{B_{\mbf{k}}} = \left.\left(\pi_B^*\mathbb{E}_{V/ B} [f]\right)\right|_{(B \times \mbf{k})_1} = \left.\left(\mbb{E}_{V\times \mbf{k}/B \times \mbf{k}} [\pi_V^* F]\right)\right|_{(B \times \mbf{k})_1} \]
where $\pi_B: B \times \mbf{k} \rightarrow B$ and $\pi_V: V \times \mbf{k} \rightarrow V$ are the projections. We then conclude by comparing the values at points with $\mathbb{E}_{V_{\mbf{k}}/B_{\mbf{k}}} [F_{V_\mbf{k}}]$ using \cref{lemma.expectation-in-families-values}.
\end{proof}

\section{Point counting motivic Euler products}\label{s.motivic-euler-products}
In this section, we define motivic Euler products for morphisms of admissible $\mathbb{Z}$-sets and then explain how to compute them using classical Euler products. 

\subsection{Motivic Euler products}
Recall from \cref{ss.admissible-z-prelim} that, for $V$ an admissible $\mathbb{Z}$-set, we have the associated $\lambda$-ring $C(V,W(\mathbb{C}))$. Recall from \cref{ss.pre-lambda-prelim} that we can then construct the filtered $\lambda$-ring of formal power series $C(V,W(\mathbb{C}))[[\ul{t}_{\mathbb{N}}]]$ with its associated $\sigma$-exponential $\Exp_\sigma$ and $\sigma$-logarithm $\Log_{\sigma}$.

\begin{definition}[Motivic Euler products]\label{def.mep}
For $V\rightarrow B$ a morphism of admissible $\mbb{Z}$-sets and 
$H \in 1 + \Fil^1 C(V,W(\mbb{C}))[[\ul{t}_{\mathbb{N}}]]$,
we define 
\[ \prod_{V/B} H := \Exp_{\sigma}\left(\int_{V/B}\Log_{\sigma}(H)\right) \in C(B, W(\mbb{C}))[[\ul{t}_{\mathbb{N}}]]. \]
When $V \rightarrow \mathbf{1}$ is the final morphism we write  
\[ \prod_{V}H= \prod_{V/\mathbf{1}}H=\Exp_{\sigma}\left(\int_V \Log_{\sigma}(H)\right).\]
\end{definition}

\begin{remark}
    The relation between this formula and the motivic Euler products of \cite{Bilu.MotivicEulerProductsAndMotivicHeightZetaFunctions} will be detailed in \cite{BiluDasHowe.SpecialValuesOfMotivicEulerProducts}, explaining the nomenclature. In the present work, we do not need to make this connection explicit, since it will suffice for our purposes to have the formula relating motivic Euler products to classical Euler products given in \cref{prop.euler-product-point-counting-formula} below (which is reproved from a different perspective in \cite{BiluDasHowe.SpecialValuesOfMotivicEulerProducts}). 
\end{remark}

\begin{example}\label{example.constant-product-is-power}
 Let $H \in W(\mathbb{C})[[\ul{t}_{\mbb{N}}]]$ with constant coefficient $1$. Then we find
 \begin{align*} \prod_V H &= \Exp_{\sigma}\left(\int_V \Log_{\sigma}(H)\right) & \textrm{ by definition} \\
 &= \Exp_{\sigma}\left(\Log_{\sigma}(H)\left(\int_V 1\right)\right) & \textrm{ by linearity of $\int_V$ }\\
 &=\Exp_{\sigma}\left(\Log_{\sigma}(H) [V]\right) & \textrm{ by \cref{example.integral-of-1-is-zeta}}\\
 &=H^{[V]} & \textrm{ by definition} \end{align*}
 as one hopes by the notation! Note that we have implicitly pulled back the coefficients of $H$ from the final object $\bf{1}$ to get an element of $C(V,W(\mbb{C}))[[\ul{t}_{\mbb{N}}]]$ --- in particular, when we are viewing $H$ as an element of $C(V,W(\mbb{C}))[[\ul{t}_{\mbb{N}}]]$ here, it is \emph{not} as the constant function on $V$ with values $H$, but rather the function that takes value $p_i \ast H$ on each degree $i$ point, where $\ast$ denotes the coefficient-wise pre-$\lambda$-ring structure as in \cref{sss.coeff-wise} (see also \cite[Example 5.3.6 and subsequent warning]{Howe.RandomMatrixStatisticsAndZeroesOfLFunctionsViaProbabilityInLambdaRings}). 
\end{example}

Motivic Euler products can be computed fiberwise (recall we defined fibers in \cref{sss.fiber}):
\begin{lemma}\label{lemma.mep-fiberwise}
Let $V\rightarrow B$ a morphism of admissible $\mbb{Z}$-sets. For any $b \in B$ and $H \in 1 + \Fil^1 C(V,W(\mbb{C}))[[\ul{t}_{\mathbb{N}}]]$,
\[ \left(\prod_{V/B} H\right)(b) = \prod_{V_b} (H|_{V_b}), \]
where evaluation at $b$ and restriction on power series are evaluated coefficient-wise.
\end{lemma}
\begin{proof}
\begin{align*} \left(\prod_{V/B} H\right)(b) & = \Exp_{\sigma}\left( \int_{V/B} \Log_\sigma(H)\right)(b) & \textrm{\parbox{4cm}{by definition}}\\
&= \Exp_{\sigma}\left( \left(\int_{V/B} \Log_\sigma(H)\right)(b)\right) & \textrm{\parbox{4cm}{because evaluation is a map of pre-$\lambda$-rings}} \\
&= \Exp_{\sigma}\left( \int_{V_b}  \Log_\sigma(H)|_{V_b}\right) & \textrm{\parbox{4cm}{because integration is defined fiberwise}}\\ 
&= \Exp_{\sigma}\left( \int_{V_b} \Log_\sigma(H|_{V_b})\right)
& \textrm{\parbox{4cm}{because the restriction is a map of pre-$\lambda$-rings}} \\
&= \prod_{V_b} H|_{V_b} & \textrm{\parbox{4cm}{by definition.}}
\end{align*}
\end{proof}

Motivic Euler products also behave well with respect to cartesian pullback. 
\begin{lemma}\label{lemma.pullback-products} Consider a cartesian diagram of admissible $\mathbb{Z}$-sets
\[\begin{tikzcd}
	{V_1} & {V_2} \\
	{B_1} & {B_2}
	\arrow["\varphi", from=1-1, to=1-2]
	\arrow[from=1-1, to=2-1]
	\arrow[from=1-2, to=2-2]
	\arrow["\psi", from=2-1, to=2-2]
\end{tikzcd}\]
For $H \in C(V_2,W(\mbb{C}))[[\ul{t}_{\mathbb{N}}]]$, 
\[ \psi^*\prod_{V_2/B_2} H = \prod_{V_1/B_1}\varphi^*H, \]
where the pullback functors are applied to power series coefficient-wise. 
\end{lemma}
\begin{proof}
We have 
\begin{align*}  \psi^*\left(\prod_{V_2/B_2} H\right) &= \psi^* \left(\Exp_{\sigma}\left(\int_{V_2/B_2} \Log_{\sigma}(H)\right)\right) & \textrm{\parbox{4cm}{by definition}} \\
& = \Exp_{\sigma}\left(\psi^* \int_{V_2/B_2} \Log_{\sigma}(H)\right) &\textrm{\parbox{4cm}{because $\psi^*$ is a map of pre-$\lambda$-rings}}\\
& = \Exp_{\sigma}\left( \int_{V_1/B_1} \varphi^* \Log_{\sigma}(H) \right) & \textrm{\parbox{4cm}{by \cref{lemma.base-change}}}\\
& = \Exp_{\sigma}\left( \int_{V_1/B_1} \Log_{\sigma}(\varphi^*H)\right) & \textrm{\parbox{4cm}{because $\varphi^*$ is a map of pre-$\lambda$-rings}}\\
& = \prod_{V_1/B_1} \varphi^*H & \textrm{\parbox{4cm}{by definition.}}  
\end{align*} 
\end{proof}

Combing these two lemmas, we deduce:
\begin{lemma}\label{lemma.mep-ext-of-k}
Let $V\rightarrow B$ be a map of admissible $\mathbb{Z}$-sets, and let $H \in C(V, W(\mathbb{C})[[\ul{t}_{\mbb{N}}]]$. For any $k \geq 1$, 
    \[ \left(\prod_{V/B} H\right)_{B_{\mbf{k}}}=\prod_{V_{\mbf{k}}/B_{\mbf{k}}} H_{V_{\mbf{k}}} \]
where the restriction $H_{V_{\mbf{k}}}$ is formed coefficient-wise as in \cref{sss.restriction-to-Vk}. 
\end{lemma} 
\begin{proof}
Identifying $B_{\mbf{k}}=(B \times \mbf{k})_1$ as in \cref{sss.ext-of-k} and using the definition in \cref{sss.restriction-to-Vk}, we apply \cref{lemma.pullback-products} to see
\[ \left(\prod_{V/B} H\right)_{B_{\mbf{k}}} = \left.\left(\pi_B^*\prod_{V/ B} H\right)\right|_{(B \times \mbf{k})_1} = \left.\left(\prod_{V\times \mbf{k}/B \times \mbf{k}} \pi_V^* H\right)\right|_{(B \times \mbf{k})_1} \]
where $\pi_B: B \times \mbf{k} \rightarrow B$ and $\pi_V: V \times \mbf{k} \rightarrow V$ are the projections. We then conclude by comparing the values at points with $\prod_{V_{\mbf{k}}/B_{\mbf{k}}} F_{V_\mbf{k}}$ using \cref{lemma.mep-fiberwise}.

\end{proof}

\subsection{Evaluation of motivic Euler products using classical Euler products}
The following proposition gives a description of motivic Euler products in terms of classical Euler products of series in $\mathbb{C}[[\ul{t}_I]]$. In light of \cref{example.constant-product-is-power}, it is a generalization of \cite[Proposition 5.2.7]{Howe.RandomMatrixStatisticsAndZeroesOfLFunctionsViaProbabilityInLambdaRings}. To state it, for $H$ a power series with coefficients in $W(\mathbb{C})$ and $i \geq 1$, we write $H_i$ for the power series $\mathbb{C}[[\ul{t}_I]]$ obtained by taking the $i$th ghost coordinate of all coefficients (note $H \mapsto H_i$ is a ring homomorphism). 

In the following, we use the convention for restriction from $V$ to $V_{\mbf{k}}$ as in \cref{sss.restriction-to-Vk} (applied coefficient-wise to power series). 
\begin{proposition}\label{prop.euler-product-point-counting-formula}
Let $V$ be an admissible $\mathbb{Z}$-set and let $H \in C(V,W(\mbb{C}))[[\ul{t}_{\mathbb{N}}]]$. 
\begin{align}\label{eq.euler-prod-k-formulas} \left(\prod_{V} H\right)_k&= \prod_{|v| \in |V_{\mbf{k}}|} H(|v|)_1(\ul{t}^{\deg(|v|)})\\
\nonumber &= \prod_{|v| \in |V|} \left(H(|v|)_{k/\gcd(k,\deg(|v|))}(\ul{t}^{\deg(|v|)/\gcd(k, \deg(|v|)}) \right)^{\gcd(k,\deg(|v|))}. \end{align} 
\end{proposition}

\begin{remark}
    Comparing the expressions in \cref{prop.euler-product-point-counting-formula} for $\left(\prod_{V} H\right)_k$ and $\left(\prod_{V_k} H\right)_1$, one finds that they agree. This is implied already by \cref{lemma.mep-ext-of-k}. 
\end{remark}

\begin{example}
    In the case that the admissible $\mathbb{Z}$-set is $Y(\overline{\mathbb{F}}_q)$ for $Y/\mathbb{F}_q$ a variety and $H$ is a power series whose coefficients are the classes associated to varieties $X_{\ul{j}}/Y$, $[X_{\ul{j}}(\fqbar)/Y(\fqbar)]$ as in \cite[Example 5.3.5]{Howe.RandomMatrixStatisticsAndZeroesOfLFunctionsViaProbabilityInLambdaRings}, \cref{prop.euler-product-point-counting-formula} says
\[ \left(\prod_{Y(\overline{\mathbb{F}}_q)} \sum_{\ul{j}} [X_{\ul{j}}(\fqbar)/V(\fqbar)] \ul{t}^{\ul{j}}\right)_k = \prod_{y \in |Y_{\mathbb{F}_{q^k}}|} \sum_{\ul{j}} \#X_{\ul{j}, \kappa(y)}(\kappa(y)) t^{\ul{j}\deg(y)}  \]
where $\kappa(y)$ denotes the residue field at the  closed point $y$ of $Y_{\mathbb{F}_{q^k}}$.
\end{example}

\begin{proof}[Proof of \cref{prop.euler-product-point-counting-formula}]
The equality of the two infinite products on the right of \cref{eq.euler-prod-k-formulas} follows from the definition of the restriction of the coefficients of $H$ in \cref{sss.restriction-to-Vk} and the fact that each orbit of degree $d$ in $V$ splits into $\gcd(d,k)$ orbits of degree $d/\gcd(d,k)$ in $V_k$. 

We now show these infinite products agree with the $k$th component of the motivic Euler product.    We first note
    \begin{align*}
   \Exp_{\sigma}\left(\int_V \Log_{\sigma}(H)\right) &= \Exp_{\sigma}\left(\sum_{|v| \in |V|} (\Log_{\sigma}(H(|v|))(t^{\deg(|v|)})\right)\\
   &= \prod_{|v| \in |V|} \Exp_{\sigma} \left(\Log_{\sigma}(H(|v|))(t^{\deg(|v|)})\right). \\
   &= \prod_{|v| \in |V|} \prod_{|v|} H|_{|v|}.\end{align*}
Thus it suffices to assume $V=\mathbf{d}$ is a single orbit. In this case we are trying to compute the $k$th component of
\[ \prod_{\bf{d}} H = \Exp_{\sigma}\left(\int_{\mbf{d}}\Log_{\sigma}(H)\right) = \Exp_{\sigma}\left(\left(\Log_{\sigma}(H(\mbf{d}))\right)(t^d)\right) \]  
where the substitution of $t^d$ for an element of $W(\mathbb{C})$ is as in \cref{sss.witt} and here it is performed coefficient-wise (note that we are writing $t$ for the variable in $W(\mathbb{C})=1+t\mathbb{C}[[t]]$ while we write $t_i$ for the power series variables!), and the computation of $\int_{\mbf{d}}$ follows from the definition \cite[Definition 5.3.2]{Howe.RandomMatrixStatisticsAndZeroesOfLFunctionsViaProbabilityInLambdaRings}. 

Thus, for $L:=\Log_{\sigma}(H(\mathbf{d}))$ and $M:= L(t^d)$,  we are trying to compute $\Exp_{\sigma}(M)$. We note that, by \cref{eq.wv-substitution-ghost}, $M_j=dL_{j/d}$ --- in particular, this is zero for $d \nmid j$. Let $\mu=\frac{\lcm(k,d)}{k}= \frac{d}{\gcd(k,d)}$ and let $\nu=\frac{\lcm(k,d)}{d}=\frac{k}{\gcd(k,d)}$. Using the expansion of $\Exp_\sigma$ in \cite[Lemma 2.5.4]{Howe.RandomMatrixStatisticsAndZeroesOfLFunctionsViaProbabilityInLambdaRings}, we find (below we note that $\ast$ is the coefficient wise pre-$\lambda$ structure as in \cref{sss.coeff-wise}):
\begin{align*} \Exp_{\sigma} \left(L(t^d)\right)_k &=  \prod_{j \geq 1} \exp \left(\frac{p_j \circ M}{j}\right)_{k}\\
&= \prod_{j \geq 1} \exp \left( \frac{(p_j\circ M)_k}{j} \right)\\
&= \prod_{j \geq 1} \exp \left( \frac{(p_j\circ p_k \ast M)_1}{j} \right)\\
&=\prod_{i \geq 1} \exp\left( \frac{ (p_{i\mu} \circ p_k \ast M)_1}{i\mu}\right).
\end{align*}
Here in the second equality we have used that passing to ghost components commutes with ring operations, and in the fourth equality we have used that $M_{n}=0$ if $d \nmid n$. Continuing by factoring out a $p_i \circ$ in each term, we obtain

\begin{align*}
&=\prod_{i \geq 1} \exp\left( \frac{ (p_{i} \circ (p_{\mu} \circ p_k \ast M))_1}{i\mu}\right)\\
&=\prod_{i \geq 1} \exp\left( \frac{ (p_{i} \circ (p_{k\mu} \ast M(\ul{t}^{\mu})))_1}{i\mu}\right) \\
&=\prod_{i \geq 1} \exp\left( \frac{ (p_i \circ(p_{k\mu/d}\ast L)(\ul{t}^{\mu})) }{i\mu/d}\right)_1 \\
&=\prod_{i \geq 1} \exp\left( \frac{d}{\mu}\frac{ p_i \circ (p_{\nu}\ast L)(\ul{t}^{\mu}) }{i}\right)_1
\end{align*}
where on the second line we have used that $p_\mu \circ$ acts as $p_{\mu} \ast$ followed by substitution of $t_i^{\mu}$ for $t_i$. Continuing by pulling out the integer multiple from the exponential as a power, we obtain
\begin{align*}
&=\prod_{i \geq 1} \exp\left( \frac{ p_i \circ (p_{\nu}\ast L)(\ul{t}^{\mu}) }{i}\right)^{\frac{d}{\mu}}_1 \\
&=\left(\Exp_{\sigma}(p_{\nu}\ast L(\ul{t}^{\mu}))^{d/\mu}\right)_1\\
&=\left(p_{\nu}\ast \Exp_{\sigma}( L)(\ul{t}^{\mu})^{d/\mu}\right)_1 \\
&=\left(p_{\nu}\ast H(\mbf{d})(\ul{t}^\mu)\right)^{d/\mu}_1 \\
&=  (H(\mbf{d})_\nu(\ul{t}^\mu))^{d/\mu}.
\end{align*}
Now, $\mathbf{d}_{\mathbf{k}}$ consists of $d/\mu=\gcd(k,d)$ orbits of degree $\mu$, and $\nu=k/\gcd({k,d})$, so we conclude. 
\end{proof}

\section{Equidistribution and independence}\label{s.equidistribution-and-independence}

In this section, we define our notion of equidistribution (\cref{def.equidistributes}), then prove our main abstract result on the computation of asymptotic moment generating functions in the presence of equidistribution, \cref{theorem.abstract-independence}. 

\subsection{Equidistributing families}
\begin{definition}\label{def.equidistributes}
Let $A \rightarrow B$ be a morphism of admissible $\mathbb{Z}$-sets admitting a section $B\rightarrow A$. Let $I$ be a directed set, and for each $d \in I$, suppose given an admissible $\mathbb{Z}$-set $U_d$ and a morphism $\ev_d: U_d \times B \rightarrow A$ of admissible $\mathbb{Z}$-sets over $B$.

\begin{itemize}
\item For any positive integer $k$ and admissible $\mbb{Z}$-subset $B' \subseteq B_\mbf{k}$, we define 
\begin{align*} \ev_{d,B'}: U_{d,\mbf{k}}( \mbf{1}) &\rightarrow \Hom_{B_{\mbf{k}}}(B',  A_{\mbf{k}})=\prod_{|b| \in |B'|}\Hom_{B_{\mbf{k}}}(|b|, A_{\mbf{k}}) \\ 
u &\mapsto \bigl(b  \mapsto (\ev_{d,\mbf{k}}(u,b) \bigr).\end{align*}
where, in the bottom formula, $u$ is viewed as an element of $U_{d,\mbf{k}}$.
\item We say $(U_d, \ev_d)$ equidistributes on $A/B$ if, for any $k \geq 1$ and any non-empty admissible $\mathbb{Z}$-subset $B' \subseteq B_{\mbf{k}}$ of finite cardinality,
\begin{equation}\label{eq.equidist}
\lim_{d \in I} \left( ({\ev_{d, B'}})_* \mu_{U_{d,\mbf{k}}(\mbf{1})} \right)= \mu_{\Hom_{B_\mbf{k}}(B', A_\mbf{k})}
\end{equation}
where, for any finite set $Z$, $\mu_Z$ is the uniform probability measure on $Z$.
\end{itemize}

\end{definition}

\begin{example}[Poonen's Bertini]
Let $Y$ be a smooth, quasi-projective subscheme of $\bP^n_{\bF_q}$ of dimension $m$.
Write $S$ for $\bF_q[x_0,\dots,x_n]$ and $S_d$ for the set of degree $d$ homogeneous polynomials in $S$.
For any field extension $L/\bF_q$, write $S(L)$ and $S_d(L)$ for $S\otimes_{\bF_q} L$ and $S_d\otimes_{\bF_q} L$, respectively.
For each point $P \in \mathbb{P}^n(\fqbar)$, we fix a $j_{P}$ such that $x_{j_{P}}$ does not vanish at $P$; we make this choice so that $j_{P}$ is constant on orbits. For $F \in S_d(\Fqbar)$, we write $F_P$ for the image of $F/x_{j_P}^d$ in $\sheaf{O}_{\bP^n_{\Fqbar},P}/\frakm_P^2$.

Set $B=\bP^n_{\bF_q}(\Fqbar)$ and for each $P\in B$, let $A_P$ be the subset of $\sheaf{O}_{\bP^n_{\Fqbar},P}/\frakm_P^2$ such that 
  \[
    A_P = 
    \begin{cases}
      \{g_P\in \sheaf{O}_{\bP^n_{\Fqbar},P}/\frakm_P^2\mid \text{image of $g_P$ in $\sheaf{O}_{Y_{\Fqbar},P}/\frakm_P^2$ is nonzero}\} & P\in Y \\
      \sheaf{O}_{\bP^n_{\Fqbar},P}/\frakm_P^2 & P\notin Y
    \end{cases}
  \]
Set $A=\bigsqcup_{P\in B}A_P$.
Let $U_d$ be the set of $F\in S_d(\Fqbar)$ such that for all $P\in B$, the image of $F$ in $\sheaf{O}_{\bP^n_{\Fqbar},P}/\frakm_P^2$ lies in $A_P$; in other words, these are the polynomials such that their scheme-theoretic vanishing set $V(F)$ intersects $Y$ transversely. 
Each of $A$, $B$, and $U_d$ are admissible $\bZ$-sets with the geometric Frobenius action.

Let $A\to B$ be the map $F_P\mapsto P$.
For an admissible $\bZ$-subset $B'\subseteq B_{\bfk}$ of finite cardinality, we have
\begin{equation*}
  \Hom_{B_{\bfk}}(B',A_{\bfk})
  = \prod_{|P|\in |B'|} A_{|P|}
\end{equation*}
where $|B'|$ is naturally viewed as a subset of $|\bP^n_{\bF_{q^k}}|$ and $A_{|P|}$ is canonically identified with a subset of  $\sheaf{O}_{\bP^n_{\bF_{q^k}},|P|}/\frakm_{|P|}^2$. 

Viewing $U_{d,\bfk}(\bfone)$ as the set of $F\in S_d(\bF_{q^k})$ such that the image of $F$ in $\sheaf{O}_{Y_{\bF_{q^k}},|P|}/\frakm_{|P|}^2$ is nonzero for all $|P|\in|\bP^n_{\bF_{q^k}}|$, the map $\ev_{d,B'}$ sends $F$ to the tuple $(F_{|P|})_{|P|\in|B'|}$.
It is a consequence of Poonen's sieve as in \cite{Poonen.BertiniTheoremsOverFiniteFields} that $(U_d, \ev_d)$ equidistributes on $A/B$; this will be explained in greater generality in \cref{ss.homog.poly}.

\end{example}

\subsection{Asymptotic \texorpdfstring{$\sigma$}{σ}-moment generating functions}

Suppose $(U_d, \ev_d)$ equidistributes in $A/B$  as in \cref{def.equidistributes}. Given a function $\mathcal{X} \in C(A, W(\mathbb{C}))$, for any $b \in B$, we obtain a random variable on the fiber $U_d \times b$ (with $\mathbb{Z}$-action multiplied by $k=\deg(b)$, i.e. $U_{d,\mbf{k}}$) by restricting $\ev_d^{-1} \mathcal{X}$. Thus we may view $\ev_d^{-1} \mathcal{X}$ as a family of random variables on $U_d$ parameterized by $B$, and then take their ``sum" by integrating over $B$ to obtain $X_d:=\int_{B \times U_d/U_d} B$. The term equidistribution suggests that, as $d\rightarrow \infty$, the random variables in the family $\ev_d^{-1} \mathcal{X}$ will behave as if they are independent, so that one expects the moment generating function of this ``sum" $X_d$ to approach the ``product" of the moment generating functions of the random variables in the family. Moreover, one expects to be able to compute the terms in this ``product": the moment generating functions of the random variable $\ev_d^{-1} \mathcal{X}|_{X \times b}$ should converge as $d \rightarrow  \infty$ to the moment generating function of $\mc{X}|_{A_b}$. The following result makes this intuition precise using motivic Euler products:

\begin{theorem}\label{theorem.abstract-independence}
   Let $A \rightarrow B$ be a morphism of admissible $\mathbb{Z}$-sets admitting a section. Suppose given a directed set $I$ and for each $d \in I$, suppose given an admissible $\mathbb{Z}$-set $U_d$ and a morphism $\ev_d: U_d \times B \rightarrow A$ of admissible $\mathbb{Z}$-sets over $B$.    If $(U_d, \ev_d)$ equidistributes on $A/B$ (\cref{def.equidistributes}) then, for any $\mc{X} \in C(A, W(\mbb{C}))$, letting 
   \[ \mc{X}_d:=\ev_d^*\mc{X} \in C(U_d \times B, W(\mbb{C})) \textrm{ and } X_d:=\int_{U_d \times B / U_d} \mc{X}_d,\]
   we have
    \begin{align}\label{eq.main-abstract-theorem-first-line} \lim_{d \in I} \mbb{E}_{U_d}[\Exp_{\sigma} (X_d h_1)] & = \lim_{d \in I} \mbb{E}_{U_d}\left[\prod_{U_d \times B / U_d} \Exp_{\sigma}(\mc{X}_d h_1)\right]\\
   \label{eq.main-abstract-theorem-second-line} & = \prod_{B} \lim_{i \in I} \mbb{E}_{U_d \times B/B}[\Exp_{\sigma}(\mc{X}_d h_1)] \\
  \label{eq.main-abstract-theorem-third-line}  & = \prod_{B} \mathbb{E}_{A/B}\left[\Exp_{\sigma} (\mc{X} h_1) \right] .\end{align}
\end{theorem}
\begin{proof} We first note that
\begin{align*} \Exp_{\sigma} (X_d h_1) &= \Exp_{\sigma} \left(\int_{U_d \times B/ U_d} \mc{X}_d h_1\right) \\
& =\Exp_{\sigma}\left(\int_{U_d \times B/ U_d}\Log_{\sigma}(\Exp_{\sigma}(\mc{X}_dh_1))\right)\\
&=\prod_{U_d \times B / U_d} \Exp_{\sigma}(\mc{X}_d h_1).\end{align*} 
In particular, from this identity we obtain the first equality \cref{eq.main-abstract-theorem-first-line}. 

To obtain the next two equalities, \cref{eq.main-abstract-theorem-second-line} and \cref{eq.main-abstract-theorem-third-line}, we argue on the $k$th component for each $k$. It follows from \cref{lemma.restriction-map} that we can compute the $k$th ghost component as
\[ \left(\mbb{E}_{U_{d}}\left[\prod_{U_d \times B / U_d} \Exp_{\sigma}(\mc{X}_d h_1)\right]\right)_k = \mbb{E}_{U_{d,\mbf{k}}(\mbf{1})}\left[ \left(\res_k \prod_{U_d \times B / U_d} \Exp_{\sigma}(\mc{X}_d h_1) \right)_1 \right]. \]

For $u \in U_{d,\bfk}(\mbf{1})$, we have 
\begin{align*} \left(\res_k\prod_{U_d \times B / U_d} \Exp_{\sigma}(\mc{X}_d h_1)\right)(u) &= \left(\prod_{U_d \times B / U_d} \Exp_{\sigma}(\mc{X}_d h_1)\right)_{U_{d,\mbf{k}}}(u) \\
&= \left(\prod_{U_{d,\mbf{k}} \times B_{\mbf{k}} / U_{d,\mbf{k}}} \Exp_{\sigma}(\mc{X}_{d, U_{d,\mbf{k}} \times B_{\mbf{k}}} h_1)\right)(u)\\
&= \prod_{B_{\mbf{k}}} \Exp_{\sigma}( \ev_{d,\mbf{k}}(u, -)^*\mc{X}_{A_\mbf{k}}h_1). 
\end{align*} 
where the first equality is by definition, the second equality is by \cref{lemma.mep-ext-of-k}, and the third equality follows from writing $\mc{X}_{d,U_{d,\mbf{k}} \times B_{\mbf{k}}}=\ev_{d,\mbf{k}}^* \mc{X}_{A_{\mbf{k}}}$ and \cref{lemma.mep-fiberwise}.

Passing to the first ghost component, we obtain, by \cref{prop.euler-product-point-counting-formula}, 
\[ \left(\res_k\prod_{U_d \times B / U_d} \Exp_{\sigma}(\mc{X}_d h_1)\right)_1(u)= \prod_{|b| \in |B_{\mbf{k}}|} \Exp_{\sigma}( \mc{X}_{A_{\mbf{k}}}(\ev_{d,\mbf{k}}(u,b))p_{\deg(|b|)})_1,\]
where to obtain the power sum monomial $p_{\deg(|b|)}$ we have used that $h_1(\ul{t}^{n})=p_n$.

Now, note that for each monomial symmetric function $m_\tau \in \Lambda$, only the finitely many $|b| \in |B_{\mbf{k}}|$ of degree less than $|\tau|$ can contribute to the coefficient of $m_\tau$ in the product on the right. Let $B'$ be the set of these, so that the coefficient of $m_\tau$ in this product is the same as its coefficient in 
\[  \prod_{|b| \in |B'|} \Exp_{\sigma}( \mc{X}_{A_{\mbf{k}}}(\ev_{d,\mbf{k}}(u,b))p_{\deg(|b|)})_1.\]

Now, since $(U_d, \ev_d)$ equidistributes on $A/B$,
\begin{multline}\label{eq.mid-af-proof-sections} \lim_{d \in I}  \mbb{E}_{U_{d,\mbf{k}}(\mbf{1})}\left[\prod_{|b| \in |B'|} \Exp_{\sigma}( \mc{X}_{A_{\mbf{k}}}(\ev_{d,\mbf{k}}(u,b))p_{\deg(|b|)})_1\right]\\
=\mathbb{E}_{\Hom_{B_{\mbf{k}}}(B', A_{\mbf{k}})} \left[\prod_{|b| \in  |B'|}\Exp_{\sigma}\left(\mc{X}_{A_{\mbf{k}}}(\varphi(|b|)) p_{\deg(|b|)}\right)_1\right], \end{multline}
where $\varphi$ on the right denotes the varying element of $\Hom_{B_{\mbf{k}}}(B',  A_{\mbf{k}})$ with respect to which we are taking expectation. Since 
\[ \Hom_{B_{\mbf{k}}}(B',  A_{\mbf{k}})=\prod_{|b| \in |B'|} \Hom_{B_{\mbf{k}}}(|b|, A_{\mbf{k}}),\]
the right-hand side of \cref{eq.mid-af-proof-sections} is equal to
\begin{equation}\label{eq.mid-af-proof-one-more} \prod_{|b| \in  |B'|}\mathbb{E}_{\Hom_{B_{\mbf{k}}}(|b|,  A_{\mbf{k}})}\left[\Exp_{\sigma}\left(\mc{X}_{A_{\mbf{k}}}(\varphi(|b|)) p_{\deg(|b|)}\right)_1\right].\end{equation}

Note that we have a natural identification
\[ \Hom_{B_{\mbf{k}}}(|b|,  A_{\mbf{k}}) = \Hom(\mbf{1}, (A_{\mbf{k}})_{b}). \]
Applying the $k=1$ case of \cref{lemma.restriction-map} to each term, we thus find that \cref{eq.mid-af-proof-one-more} is equal to
\begin{equation}
\prod_{|b| \in |B'|} \left(\mbb{E}_{(A_{\mathbf{k}})_b}\left[\Exp_{\sigma}(\mc{X}_{A_{\mathbf{k}}}|_{(A_{\mathbf{k}})_b} p_{\deg(|b|)})\right] \right)_1. \end{equation}

By \cref{lemma.expectation-in-families-values}, this is equal to

\begin{equation}\label{eq.mid-af-proof-again}
\prod_{|b| \in |B'|} \left(\mbb{E}_{A_\mbf{k}/B_{\mbf{k}} }\left[\Exp_{\sigma}(\mc{X}_{A_{\mbf{k}}} p_{\deg(|b|)})\right](|b|) \right)_1. \end{equation}
Note that, again for the coefficient of any fixed $m_\tau$,  \cref{eq.mid-af-proof-again} agrees with the corresponding product over $|B_{ \mbf{k}}|$ whenever $|B'|$ is a sufficiently large subset of finite cardinality. Putting this all together, we obtain the desired equality: 
\begin{align*} \lim_{d \in I} \left(\mbb{E}_{U_d}\left[\prod_{U_d \times B / U_d} \Exp_{\sigma}(\mc{X}_d h_1)\right]\right)_k &= \prod_{|b| \in |B_{\mbf{k}}|} \left(\mbb{E}_{A_{\mbf{k}}/B_{\mbf{k}}}\left[\Exp_{\sigma}(\mc{X}_{A_{\mbf{k}}} p_{\deg(|b|)})\right](|b|) \right)_1 \\
&=\prod_{|b| \in |B_{\mbf{k}}|} \left(\left(\mathbb{E}_{A/B}\left[ \Exp_{\sigma}(\mc{X} p_{\deg(|b|)})\right]\right)_{B_{\mbf{k}}} (|b|) \right)_1 \\
&= \left(\prod_{B} \mathbb{E}_{A/B}\left[ \Exp_{\sigma}(\mc{X}h_1) \right] \right)_k\end{align*}
where the second equality is by \cref{lemma.expectations-change-of-k} and the third is by \cref{prop.euler-product-point-counting-formula}. 
\end{proof}

\begin{remark}\label{remark.more-general-moment-generating-functions-and-joint}
    Let $H \in 1 + \Fil^1\mathbb{Z}[[\ul{t}_{\mathbb{N}}]]$. Then, either by the same argument as in the proof, or as a formal consequence of the associated equivalence of $\Lambda$-distributions (cf. \cite[\S3.2]{Howe.RandomMatrixStatisticsAndZeroesOfLFunctionsViaProbabilityInLambdaRings}), in the setting of \cref{theorem.abstract-independence} one also obtains:    \begin{align*}\lim_{d \in I} \mbb{E}_{U_d}[H^{X_d}] & = \lim_{d \in I} \mbb{E}_{U_d}\left[\prod_{U_d \times B / U_d} H^{X_d} \right]\\
  & = \prod_{U_d \times B / U_d} \lim_{i \in I} \mbb{E}_{U_d \times B/B}[H^{X_d}] \\
  & = \prod_{B} \mathbb{E}_{A/B}\left[H^{\mc{X}} \right].\end{align*}
The $\sigma$-moment generating function corresponds to $H=\Exp_{\sigma}(h_1)=1+h_1+h_2+\ldots$, while the falling moment generating function of \cite[Example 3.2.3-(2)]{Howe.RandomMatrixStatisticsAndZeroesOfLFunctionsViaProbabilityInLambdaRings} corresponds to $H=1+h_1$. Similarly, the proof can be extended to show one can compute joint moment generating functions: for $\mathcal{X}$ and $\mc{Y}$ two families,
\[ \lim_{d \in I} \mathbb{E}_{U_d}\left[ H(\ul{t})^{X_d} H(\ul{s})^{Y_d}\right] = \prod_{B} \mathbb{E}_{A/B}\left[ H(\ul{t})^{\mc{X}} H(\ul{s})^{\mc{Y}} \right]. \]
\end{remark}

\section{Equidistribution for homogeneous polynomials}\label{s.homog}

In this section we first use  \cite[Theorem C]{Bertucci.TaylorConditionsOverFiniteFields} to show equidistribution holds for homogeneous polynomials with certain natural conditions on their Taylor expansions (generalizing those accessible using \cite[Theorem 1.3]{Poonen.BertiniTheoremsOverFiniteFields}) --- see \cref{prop.homog-equidistribution}. In \cref{ss.example-L-functions-of-characters} we explain how the computation of the $\sigma$-moment generating functions for $L$-functions of Dirichlet characters made in \cite[Theorem B]{Howe.RandomMatrixStatisticsAndZeroesOfLFunctionsViaProbabilityInLambdaRings} can be handled by combining \cref{prop.homog-equidistribution} and \cref{theorem.abstract-independence}. In \cref{ss.exotic} we apply \cref{prop.homog-equidistribution} and \cref{theorem.abstract-independence} to compute the asymptotic $\sigma$-moment generating functions for zeta functions of hypersurfaces in a quasi-projective variety satisfying some exotic transversality conditions as in \cite[Example 5.2]{Bertucci.TaylorConditionsOverFiniteFields} --- see \cref{theorem.exotic-tranversality-zeta}.

\subsection{Establishing equidistribution}
\label{ss.homog.poly}

Let $Y_1,\dots,Y_u$ be quasi-projective subschemes of $\bP^n_{\bF_q}$ of dimensions $m_i=\dim Y_i$ with locally closed embeddings $\iota_1,\dots,\iota_u$, respectively. For each $i$, let $\sheaf{Q}_i$ be a locally free quotient of $\iota_i^*\Omega^1_{\bP^n}$ of rank $\ell_i\geq m_i$, and let $\sheaf{K}_i=\ker(\iota_i^*\Omega^1_{\bP^n}\to\sheaf{Q}_i)$. Define
  \[
    \sheaf{E}_i=\bigl(\iota_i^*\PP^1(\sheaf{O}_{\bP^n_{\bF_q}})\bigr)/\sheaf{K}_i
    \quad
    \text{and}
    \quad
    \sheaf{E}_{i,d}=\bigl(\iota_i^*\PP^1(\sheaf{O}_{\bP^n_{\bF_q}}(d))\bigr)/\sheaf{K}_i(d)
  \]
where $\PP^1(\sheaf{F})$ is the sheaf of 1-principal parts of an $\sheaf{O}_{\bP^n}$-module $\sheaf{F}$.

We write $x_0, \ldots, x_n$ for the homogeneous coordinates on $\mathbb{P}^n$. For every point $P\in\bP^n(\fqbar)$, fix a nonnegative integer $M_P$ and a $j_p$ such that $x_{j_P}$ is non-vanishing at $P$; we make these choices so that $M_P$ and $j_P$ are constant on orbits in $\bP^n(\fqbar)$. Given $F\in S_d(\Fqbar)$, write $F_P$ for the image of $F/x_{j_P}^d$ in $\sheaf{O}_{\bP^n_{\Fqbar},P}/\frakm_P^{M_P+1}$.

Let $\phi:\bP^n_{\Fqbar}\to\bP^n_{\bF_q}$ be the natural map. 

\begin{proposition}\label{prop.homog-equidistribution}
With notation as above, set $B=\bP^n_{\bF_q}(\Fqbar)$. For each $P\in B$, let $A_P$ be a subset of $\sheaf{O}_{\bP^n_{\Fqbar},P}/\frakm_P^{M_P+1}$ such that for all but finitely many $P$, $A_P$ contains $F_P$ for all homogeneous $F\in S(\Fqbar)$ such that the image of $F_P$ in $\restr{\phi^*\sheaf{E}_i}{P}$ is nonzero for all $i$.

Set $A=\bigsqcup_{P\in B}A_P$.
Let $U_d$ be the set of $F\in S_d(\Fqbar)$ such that for all $P\in B$, the image $F_P$ of $F$ in $\sheaf{O}_{\bP^n_{\Fqbar},P}/\frakm_P^{M_P+1}$ lies in $A_P$.
Each of $A$, $B$, and $U_d$ are admissible $\bZ$-sets with the geometric Frobenius action.

Let $A\to B$ be the map $F_P\mapsto P$ and $\ev_d:U_d\times B\to A$ the map $(F,P)\mapsto F_P$. Then $(U_d,\ev_d)$ equidistributes on $A/B$ in the sense of \cref{def.equidistributes}.
\end{proposition}

\begin{proof}
For an admissible $\bZ$-subset $B'\subseteq B_{\bfk}$ of finite cardinality, we have
\begin{equation*}
  \Hom_{B_{\bfk}}(B',A_{\bfk})
  = \prod_{|P| \in |B'|} \Hom_{B_{\bfk}}(|P|, A_{\bf{k}}).
\end{equation*}
If we identify each orbit $|P| \subseteq B' \subseteq \mathbb{P}^n(\fqbar)$ with a closed point of $\mathbb{P}^n_{\mathbb{F}_{q^k}}$, then we obtain a canonical identification of 
$\Hom_{B_{\bfk}}(|P|, A_{\bf{k}})$ with a subset $A_{|P|}$ of  $\sheaf{O}_{\bP^n_{\mathbb{F}_{q^k}},|P|}/\frakm_{|P|}^{M_P+1}$. Viewing $U_{d,\bfk}(\bfone)$ as the set of $F\in S_d(\bF_{q^k})$ such that the image $F_{|P|}$ of $F/x_{j_{|P|}}^d$ in $\sheaf{O}_{\bP^n_{\bF_{q^k}},|P|}/\frakm_{|P|}^{M_{|P|}+1}$ lies in $A_{|P|}$ for all $|P|\in|\bP^n_{\bF_{q^k}}|$, the map $\ev_{d,B'}$ sends $F$ to the tuple $(F_{|P|})_{|P|\in|B'|}$.

To establish \cref{eq.equidist} it suffices to show equality for each singleton $\{(F_{|P|})_{|P|}\}$.
On the right side we have
\begin{equation*}
  \mu_{\Hom_{B_{\bfk}}(B',A_{\bfk})}(\{(F_{|P|})_{|P|}\})
  = \prod_{|P|\in|B'|} \frac{1}{\# A_{|P|}}.
\end{equation*}
The left side can be viewed as a conditional probability:
\begin{align*}
  \MoveEqLeft \bigl((\ev_{d,B'})_*\mu_{U_{d,\bfk}(\bfone)}\bigr)(\{(F_{|P|})_{|P|}\}) \\
  &= \mu_{U_{d,\bfk}(\bfone)}(\ev_{d,B'}^{-1}(\{(F_{|P|})_{|P|}\}) \\
  &= \frac{\#\{G\in S_d(\bF_{q^k})\mid G_{|P|}=F_{|P|} \text{ for all $|P|\in |B'|$ and $G\in U_{d,\bfk}(\bfone)$}\}/\# S_d(\bF_{q^k})}{\# U_{d,\bfk}(\bfone)/\# S_d(\bF_{q^k})}.
\end{align*}
By \cite[Theorem C]{Bertucci.TaylorConditionsOverFiniteFields}, as $d\to\infty$, this converges to
\begin{equation*}
  \frac{
    \biggl(
      \prod_{|P|\in|B'|}
      \frac{1}{\#\bigl(\sheaf{O}_{\bP^n_{\bF_{q^k}}}/\frakm_{|P|}^{M_{|P|}+1}\bigr)}
    \biggr)
    \biggl(
      \prod_{|P|\notin|B'|}
      \frac{\# A_{|P|}}{\#\bigl(\sheaf{O}_{\bP^n_{\bF_{q^k}}}/\frakm_{|P|}^{M_{|P|}+1}\bigr)}
    \biggr)
  }{
    \prod_{|P|\in|\bP^n_{\bF_{q^k}}|}
      \frac{\# A_{|P|}}{\#\bigl(\sheaf{O}_{\bP^n_{\bF_{q^k}}}/\frakm_{|P|}^{M_{|P|}+1}\bigr)}
   }
  = \prod_{|P|\in|B'|}\frac{1}{\# A_{|P|}}.
\end{equation*}
So \cref{eq.equidist} is satisfied, and thus $(U_d,\ev_d)$ equidistributes on $A/B$.
\end{proof}

\begin{remark}
When each of the $Y_i$ is smooth and $\sheaf{Q}_i=\Omega_{Y_i}$, one can invoke  \cite[Theorem 1.3]{Poonen.BertiniTheoremsOverFiniteFields} instead of its generalization \cite[Theorem C]{Bertucci.TaylorConditionsOverFiniteFields}.
\end{remark}

\subsection{Example: \texorpdfstring{$L$}{L}-functions of characters}\label{ss.example-L-functions-of-characters}

We now explain how this framework leads to a more transparent computation of the asymptotic $\sigma$-moment generating functions for $L$-functions of Dirichlet characters as in \cite[Theorem B]{Howe.RandomMatrixStatisticsAndZeroesOfLFunctionsViaProbabilityInLambdaRings}.

Let $\kappa$ be a finite field of cardinality $q$ and fix a prime $\ell$ coprime to $q$, and a non-trivial order $\ell$ character $\chi$ of $\mu_\ell(\kappa)$. Let $U_d$ be the space of $\ell$-power free degree $d$  polynomials in the variable $x$. As in \cite[\S7.1]{Howe.RandomMatrixStatisticsAndZeroesOfLFunctionsViaProbabilityInLambdaRings}, for $f \in U_d$, we obtain a Kummer character $\chi_f$ of $\Gal(\overline{\kappa(f)(x)}/\kappa(f)(x))$ sending $\sigma$ to $\chi(\sigma(f^{1/\ell})/f^{1/\ell})$. One can compute the $L$-function as an Euler product 
\begin{equation}\label{eq.char-euler-product} \mathcal{L}(\chi_f, t) = \prod_{|z| \in |\mathbb{A}^1_{\kappa(f)}|} \frac{1}{1-\chi(f(z)^{(\#\kappa(f,z)-1)/\ell})t^{\deg(z)}} \end{equation}
where we set $\chi(0)=0$. We define a random variable $X_d$ on $U_d$ sending $f$ to $\mathcal{L}(\chi_f, t)$. 

We explain how this fits into the context of \cref{prop.homog-equidistribution}: for each $P \in \mathbb{A}^1(\kappabar)\subseteq \mathbb{P}^1(\kappabar)$, we set $M_P=\ell-1$, $j_P=1$, and $A_P= \sheaf{O}_{\mathbb{P}^n_{\kappabar}, P}/\mathfrak{m}_P^\ell - \{0\}.$ For $\infty=[1:0]$, we set $M_\infty=0$, $j_\infty=0$, and $A_\infty=\{1\} \subseteq \sheaf{O}_{\mathbb{P}^n_{\kappabar},\infty}/\mathfrak{m}_\infty = \kappabar$. Then dividing by $x_1^d$ identifies the $U_d$ appearing in \cref{prop.homog-equidistribution} with the set of $\ell$-power free monic polynomials in $x=x_0/x_1$ that we have called $U_d$ here, and under this identification the map $\ev_d$ is the natural map.

Now, for $P \in \mathbb{A}^1(\kappabar)$, we define $\mathcal{X}_P$ to send a germ $g \in A_P$ to 
\[ \frac{1}{1-\chi(g(P)^{(\#\kappa(g)-1)/\ell})t} = [\chi(g(P)^{(\#\kappa(g)-1)/\ell})]\]
where $\kappa(g)$ is the extension of $\kappa(P)$ in $\kappabar$  generated by the coefficients of $g$. We define $\mathcal{X}_\infty$ to be the trivial random variable. 
It is a straightforward computation from \cref{eq.char-euler-product} that 
\[ X_d = \int_{U_d \times \mathbb{A}^1(\kappabar)/U_d} \ev^*\mathcal{X},\]
thus \cref{prop.homog-equidistribution} and \cref{theorem.abstract-independence} imply
\[ \lim_{d \rightarrow \infty} \mathbb{E}[\Exp_{\sigma}(X_d h_1)]=\prod_{\mathbb{A}^1(\kappabar)} \mathbb{E}_{A/\mathbb{A}^1(\kappabar)}[\Exp_\sigma(\mathcal{X}_P h_1)], \]
where the term at $\infty$ has gone away because it is identically $1$. 

Now, we compute $\mathbb{E}[\Exp_{\sigma}(\mathcal{X}_P h_1)]$ for $P \in \mathbb{A}^1(\kappabar)$: let $\chi_P$ denote the $\mathbb{C}$-valued function on $A_P$ sending a germ $g$ to $\chi(g(P)^{(\#\kappa(g)-1)/\ell})$ so that $\mathcal{X}_P=[\chi_P]$. Now, since $h_j \circ ([a]h_1)= [a^j] h_j$ for any $a \in \mathbb{C}$ (see \cite[Lemma 2.2.4]{Howe.RandomMatrixStatisticsAndZeroesOfLFunctionsViaProbabilityInLambdaRings}),
\[ \Exp_{\sigma}(\mathcal{X}_P h_1)= \sum_{j \geq 0} [\chi_P^{j}] h_{j}. \]
We can compute the $k$th component of $\mathbb{E}[\chi_P^n]$ using \cref{lemma.restriction-map}. The restriction of $[\chi_P]$ to $(A_P)_{\bfk}(\bf{1})$, i.e. to the set of germs with coefficients in the degree $k$ extension of $\kappa(P)$ in $\kappabar$,  
sends $g$ to $[\chi(g(P)^{(q_P^{k}-1)/\ell})]$, where $q_P=\#\kappa(P)$. Taking the first component, we find, $[\chi_P^n]_1=[\chi_P]^n_1$ sends $g$ to 
\[ \chi(g(P)^{n(q_P^{k}-1)/\ell}). \]
Thus the expectation is zero unless $\ell | n$ (since if $\ell \nmid n$ then every $\ell$th root of unity value is equally likely and the only other value it takes is zero), and when $\ell|n$ it is 
\[ \frac{(q_P^{k})^{\ell-1}-1}{(q_P^k)^\ell -1} \]
since the function is identically $1$ when $g(P)\neq 0$ and $0$ when $g(P)=0$. Thus,
\[ \mathbb{E}[\Exp_{\sigma}(\mathcal{X}_P h_1)]= 1 + \frac{([q_P])^{\ell-1}-1}{([q_P])^\ell -1} \sum_{j \geq 1}  h_{\ell j}.  \]
It follows that  
\[ \mathbb{E}_{A/\mathbb{A}^1(\kappabar)}[\Exp_{\sigma}(\mathcal{X} h_1)] \]
is the pullback from a point of  
\[  1 + \frac{[q]^{\ell-1}-1}{[q]^\ell -1} \sum_{j \geq 1}  h_{\ell j}, \]
thus
\begin{align*} \prod_{\mathbb{A}^1(\kappabar)} \mathbb{E}_{A/\mathbb{A}^1(\kappabar)}[\Exp_\sigma(\mathcal{X} h_1)] & = \Bigl(1 + \frac{[q]^{\ell-1}-1}{[q]^\ell -1} \sum_{j \geq 1}  h_{\ell j}\Bigr)^{\mathbb{A}^1(\kappabar)}\\
&=\Bigl(1 + \frac{[q]^{\ell-1}-1}{[q]^\ell -1} \sum_{j \geq 1}  h_{\ell j}\Bigr)^{[q]}. \end{align*}
In particular, if we
\begin{enumerate}
    \item  replace $\mathcal{L}$ with its reciprocal (as a power series in $1+t\mathbb{C}[[t]]$, which is the negative for the additive structure of the Witt vectors), and
    \item scale the variables by $[q^{-1/2}]$,
\end{enumerate}
then we recover the $\sigma$-moment generating function described in \cite[Theorem B]{Howe.RandomMatrixStatisticsAndZeroesOfLFunctionsViaProbabilityInLambdaRings}. Indeed, by \cite[Theorem 2.2.1]{Howe.TheNegativeSigmaMomentGeneratingFunction}, replacing $\mathcal{L}$ with its reciprocal will replace the $h_{\ell j}$'s with $(-1)^{\ell j}e_{\ell j}$'s in the $\sigma$-moment generating function, and scaling the random variable by $[z]$ changes the $\sigma$-moment generating function by scaling the variables by $[z]$.

When $\ell \geq 3$, the more refined joint moment generating function between $X_d$ and the random variable $\overline{X}_d$ obtained from the complex conjugate character to $\chi$ is of the most interest; this can be recovered similarly by using the natural extension of \cref{theorem.abstract-independence} to joint moment generating functions described in \cref{remark.more-general-moment-generating-functions-and-joint}.

\subsection{Application: Zeta functions of hypersurface sections with exotic transversality conditions}\label{ss.exotic}
We consider the setup of \cite[Example 5.2]{Bertucci.TaylorConditionsOverFiniteFields}. Let $\kappa$ be a finite field of cardinality $q$, let $Y \hookrightarrow \mathbb{P}^m_{\kappa}$ be a quasi-projective subscheme, let $W \hookrightarrow Y \times_{\kappa} \mathbb{P}^{m}_{\kappa}$ be a closed subscheme such that the projection from $W$ to $Y$ is smooth of relative dimension $\ell \geq \dim Y$ and such that the graph of $Y \hookrightarrow \mathbb{P}^m_{\kappa}$, $Y \rightarrow Y \times_{\kappa} \mathbb{P}^m_{\kappa}$, factors through $W$. In other words, $W \rightarrow Y$ is a smooth family of subvarieties of $\mathbb{P}^m_{\kappa}$ of at least the same dimension as $Y$ such that, at each point $P \in Y(\kappabar)$, the fiber $W_P$ contains $P$. 

We let $U_d$ be the set of degree $d$ homogeneous polynomials $F$ in $m+1$ variables such that $V(F)$ is transverse to $W_P$ at all points $P \in Y(\kappabar) \cap V(F)(\kappabar)$. We write $Z_d$ for the random variable on $U_d$ sending $F$ to the zeta function
\[ Z_{V(F) \cap Y_{\kappa(F)}/\kappa(F)}(t). \]

\begin{example}\label{example.families-conditions}
    If $Y$ is smooth and $W=Y \times \mathbb{P}^n_{\kappa}$, then $Z_d$ is the random variable sending a smooth hypersurface section to its zeta function. On the other hand, \cite[Example 5.2]{Bertucci.TaylorConditionsOverFiniteFields} shows there are other geometrically interesting examples. 
\end{example}

In light of \cref{example.families-conditions} and \cite[Theorem 2.2.1]{Howe.TheNegativeSigmaMomentGeneratingFunction}, the following is a generalization of \cite[Theorem 8.3.1]{Howe.RandomMatrixStatisticsAndZeroesOfLFunctionsViaProbabilityInLambdaRings}.

\begin{theorem}\label{theorem.exotic-tranversality-zeta}
With notation as above, as $d \rightarrow \infty$, the $\Lambda$-distribution of $Z_{d}$ converges to a binomial $\Lambda$-distribution as in \cite[Definition 3.3.2]{Howe.RandomMatrixStatisticsAndZeroesOfLFunctionsViaProbabilityInLambdaRings} with parameters
\[ p=\frac{[q^{-1}]-[q^{-(\ell+1)}]}{1- [q]^{-(\ell+1)}}=\frac{[q^{\ell}]-1}{[q^{\ell+1}]-1} \textrm { and } N=[Y(\overline{\kappa})]. \]
Equivalently, 
\[ \lim_{d \rightarrow \infty} \mathbb{E}[\Exp_{\sigma}(Z_{d}h_1)] = (1+p(h_1 + h_2 + \ldots))^N. \]    
\end{theorem}
\begin{proof}
We are in the setup of \cref{prop.homog-equidistribution} with $M_P=1$ for all $P$, with $u=1$ and $\sheaf{Q}=\sheaf{Q}_1 = \gamma^*\Omega_{W/Y}$, where $\gamma$ is the map $Y \rightarrow W$ induced by the graph of the immersion $Y \rightarrow \mathbb{P}^m_{\kappa}$. The $j_P$ can be chosen arbitrarily; it will not effect the conditions below.

For $P \in Y(\kappabar)$, we consider the random variable $\mathcal{X}_P$ on $A_P$ sending a germ $g$ to $\frac{1}{1-t}$ if $g(P)=0$ and $0$ otherwise, and for $P \in \mathbb{P}^m(\kappabar)\backslash Y(\kappabar)$, we set $\mathcal{X}_P$ to be the trivial random variable.    

Then 
\[ Z_d= \int_{U_d \times \mathbb{P}^m(\kappabar)/ U_d} \ev_d^{*} \mathcal{X}. \]
By \cref{prop.homog-equidistribution}, $(U_{d}, \ev_{d})$ equidistributes on $A/\mathbb{P}^m(\kappabar)$, thus, by \cref{theorem.abstract-independence}, the asymptotic $\sigma$-moment generating function is
\[ \prod_{\mathbb{P}^m(\kappabar)} \mathbb{E}_{A/\mathbb{P}^m(\kappabar)}[\Exp_{\sigma}(\mathcal{X}h_1)] = \prod_{Y(\kappabar)} \mathbb{E}_{A_{Y(\kappabar)}/Y(\kappabar)}[\Exp_{\sigma}(\mathcal{X}h_1)] \]
where the factor corresponding to $\mathbb{P}^m(\kappabar)-Y(\kappabar)$ disappeared because the moment generating function of the trivial random variable is identically $1$. 

For $P \in Y(\kappabar)$ of degree $k$, $\mathcal{X}_P$ is a Bernoulli random variable equal to $\frac{1}{1-t}$ (the unit in $W(\mathbb{C})$) with probability 
\[ \frac{[q^{-k}] - [q^{-(\ell+1)k}]}{1- [q^{-(\ell+1)k}]}. \]
This can be checked on each ghost component using \cref{lemma.restriction-map}; the numerator is the probability of a section vanishing at $P$ and being transverse to $W_P$ at $P$, whereas the denominator is the probability of a section either not vanishing or vanishing and being transverse. 

It follows that $\Exp_{\sigma}(\mathcal{X}h_1)|_{Y(\kappabar)}$ is the pullback of
$(1+ p(h_1 + h_2 + \ldots))$ from $\mathbf{1}$ to $Y(\kappabar)$. Thus, by \cref{example.constant-product-is-power},
\[ \prod_{Y(\kappabar)} \mathbb{E}_{A_{Y(\kappabar)}/Y(\kappabar)}[\Exp_{\sigma}(\mathcal{X}h_1)] = \prod_{Y(\kappabar)} (1+ p(h_1 + h_2 + \ldots)) = (1+ p(h_1 + h_2 + \ldots))^{[Y(\kappabar)]}.   \] 
\end{proof}

\begin{remark}
    For $F \in U_d$ as above, $V(F)$ may not intersect $Y$ transversely, so there is not an obvious $L$-function to extract in order to give a result analogous to \cref{maintheorem.complete-intersections}. However, \cref{prop.homog-equidistribution} is robust enough to allow one to additionally impose the condition that $V(F)$ intersect $Y$ transversely, giving an interesting $L$-function. We leave the computation of the $L$-function $\Lambda$-distribution in this case to the interested reader, but note that the motivic Euler product for the zeta function random variable in this setting will not typically be expressible as a single pre-$\lambda$ power:  the pointwise moment generating function at $P$ will depend on the dimension of the intersection of the tangent space of $Y$ at $P$ and the tangent space of $W_P$ at $P$.   
\end{remark}

\section{Equidistribution for tuples of homogeneous polynomials}\label{s.tuples-homog}
In this section we use the results of \cite{BucurKedlaya.TheProbabilityThatACompleteIntersectionIsSmooth} to show equidistribution holds for tuples of homogeneous polynomials intersecting a fixed smooth quasi-projective variety transversely (\cref{prop.tuples-equidistribution}), then combine this with \cref{theorem.abstract-independence} to prove \cref{maintheorem.complete-intersections}. As in the proof of \cite[Theorem C]{Howe.RandomMatrixStatisticsAndZeroesOfLFunctionsViaProbabilityInLambdaRings}, we first study the geometric random variable that sends a complete intersection to its zeta function (\cref{theorem.ci-geometric-rv}; cf. \cite[Theorem 8.3.1]{Howe.RandomMatrixStatisticsAndZeroesOfLFunctionsViaProbabilityInLambdaRings}), then argue with basic properties of independence to obtain the computation of the asymptotic moment generating function in \cref{maintheorem.complete-intersections}. 

\subsection{Establishing equidistribution}
\label{ss.Bucur-Kedlaya}

\subsubsection{}
For $\ul{d}=(d_1,\dots,d_r)$ a tuple of positive integers, write $S_{\ul{d}}$ for the product $S_{d_1}\times\dots\times S_{d_r}$ and identify it with the global sections of $\sheaf{O}_{\bP^n_{\bF_q}}(\ul{d})=\bigoplus_{i=1}^r\sheaf{O}_{\bP^n_{\bF_q}}(d_i)$.

\subsubsection{}\label{sss.tuples-ordering}Fix non-negative integers $m$ and $r$. Set $I_{m,r}=\bN^r$ with an ordering described as follows: for tuples $\ul{a}=(a_1,\dots,a_r)$ and $\ul{b}=(b_1,\dots,b_r)$, $\ul{a}\leq\ul{b}$ if and only if $a_i\leq b_i$ for all $i$ and $\max(b_i)^{-m}q^{\min(b_i)/(m+1)}\leq \max(a_i)^{-m}q^{\min(a_i)/(m+1)}$.

\subsubsection{}For every point $P\in\bP^n(\fqbar)$, fix a non-vanishing coordinate $x_{j_P}$, $0 \leq j_p \leq n$; we make this choice so that $j_P$ is constant on orbits. 
Given $\ul{F}\in S_{\ul{d}}(\fqbar)$, write $\ul{F}_P$ for the image of 
$(F_1/x_{j_P}^{d_1},\dots,F_r/x_{j_P}^{d_r})$ in $(\sheaf{O}_{\bP^n_{\fqbar},P}/\frakm_P^2)^{\oplus r}$. 

\subsubsection{} Define
  \[
    L(a,b,c) = \prod_{j=0}^{c-1} (1-a^{-(b-j)})
  \]
which, when $a=q$, is the probability that $c$ randomly chosen vectors in $\bF_q^b$ are linearly independent.

\begin{proposition}\label{prop.tuples-equidistribution}
Let $Y$ be a smooth, quasi-projective subscheme of $\bP^n_{\bF_q}$ of dimension $m$.

Set $B=\bP^n_{\bF_q}(\Fqbar)$ and for each $P\in B$, let $A_P$ be the subset of $(\sheaf{O}_{\bP^n_{\Fqbar},P}/\frakm_P^2)^{\oplus r}$ such that
  \[
    A_P = 
    \begin{cases}
      \biggl\{\ul{g}_P\in (\sheaf{O}_{\bP^n_{\Fqbar},P}/\frakm_P^2)^{\oplus r}\biggm\vert \begin{tabular}{@{}l@{}}
      not all $g_i$ lie in $\frakm_P$ or all $g_i$ lie in $\frakm_P$ \\
      and are linearly independent in $\frakm_P/\frakm_P^2$
      \end{tabular}
      \biggr\} & P\in Y \\
      (\sheaf{O}_{\bP^n_{\Fqbar},P}/\frakm_P^2)^{\oplus r} & P\notin Y
    \end{cases}
  \]
Set $A=\bigsqcup_{P\in B}A_P$.
Let $U_{\ul{d}}$ be the set of $\ul{F}\in S_{\ul{d}}(\Fqbar)$ such that for all $P\in B$, $\ul{F}_P$ lies  in $A_P$. Let $A\to B$ be the map $\ul{F}_P\mapsto P$ and $\ev_{\ul{d}}:U_{\ul{d}}\times B\to A$ the map $(\ul{F},P)\mapsto \ul{F}_P$. Then $(U_{\ul{d}},\ev_{\ul{d}})_{\ul{d} \in I_{m,r}}$ equidistributes on $A/B$ in the sense of  \cref{def.equidistributes}. 
\end{proposition}

\begin{proof}
For an admissible $\bZ$-subset $B'\subseteq B_{\bfk}$ of finite cardinality, we have
\begin{equation*}
  \Hom_{B_{\bfk}}(B',A_{\bfk})
  = \prod_{|P| \in |B'|} \Hom_{B_{\bfk}}(|P|, A_{\bf{k}}).
\end{equation*}
If we identify the orbit $|P| \subseteq B' \subseteq \mathbb{P}^n(\fqbar)$ with a closed point of $\mathbb{P}^n_{\mathbb{F}_{q^k}}$, then we obtain a canonical identification of 
$\Hom_{B_{\bfk}}(|P|, A_{\bf{k}})$ with a subset $A_{|P|}$ of  $(\sheaf{O}_{\bP^n_{\mathbb{F}_{q^k}},|P|}/\frakm_{|P|}^2)^{\oplus r}$.
Viewing $U_{\ul{d},\bfk}(\bfone)$ as the set of $\ul{F}\in S_{\ul{d}}(\bF_{q^k})$ such that the image of $\ul{F}$ in $(\sheaf{O}_{\bP^n_{\bF_{q^k}},|P|}/\frakm_{|P|}^2)^{\oplus r}$ lies in $A_{|P|}$ for all $|P|\in|\bP^n_{\bF_{q^k}}|$, the map $\ev_{\ul{d},B'}$ sends $\ul{F}$ to the tuple $(\ul{F}_{|P|})_{|P|\in|B'|}$.

To establish \cref{eq.equidist} it suffices to show equality for each singleton $\{(\ul{F}_{|P|})_{|P| \in |B'|}\}$.
On the right side we have
\begin{align*}
  \MoveEqLeft
  \mu_{\Hom_{B_{\bfk}}(B',A_{\bfk})}(\{(\ul{F}_{|P|})_{|P|}\}) \\
  &= \prod_{|P|\in|B'|} \frac{1}{\# A_{|P|}} \\
  &= \Bigl(\prod_{|P|\in|B'|-|Y_{\bF_{q^k}}|} q^{-kr\deg(P)(n+1)}\Bigr) \\
  &\hspace{1.5cm}  \Bigl(\prod_{|P|\in|B'|\cap|Y_{\bF_{q^k}}|} \frac{q^{-kr\deg(P)(n+1)}}{1-q^{-kr\deg(P)}+q^{-kr\deg(P)}L(q^{k\deg(P)},m,r)}\Bigr).
\end{align*}
The left side of \cref{eq.equidist} can be viewed as a conditional probability:
\begin{align*}
  \MoveEqLeft \bigl((\ev_{\ul{d},B'})_*\mu_{U_{\ul{d},\bfk}(\bfone)}\bigr)(\{(\ul{F}_{|P|})_{|P|}\}) \\
  &= \mu_{U_{\ul{d},\bfk}(\bfone)}(\ev_{\ul{d},B'}^{-1}(\{(\ul{F}_{|P|})_{|P|}\}) \\
  &= \frac{\#\{\ul{G}\in S_{\ul{d}}(\bF_{q^k})\mid \ul{G}_{|P|}=\ul{F}_{|P|} \text{ for all $|P|\in |B'|$ and $\ul{G}\in U_{\ul{d},\bfk}(\bfone)$}\}/\# S_{\ul{d}}(\bF_{q^k})}{\# U_{\ul{d},\bfk}(\bfone)/\# S_{\ul{d}}(\bF_{q^k})}.
\end{align*}

By \cite[Theorem 1.2]{BucurKedlaya.TheProbabilityThatACompleteIntersectionIsSmooth}, as $\ul{d}\to\infty$ such that $\mr{max}(\ul{d})^mq^{-\min(\ul{d})/(m+1)}\to 0$ (guaranteed by our choice of $I_{m,r}$), this converges to
\begin{multline*}
  \frac
    {\Bigl(\prod_{|P|\in|B'|}q^{-kr\deg(P)(n+1)}\Bigr)\Bigl(\prod_{|P|\in|Y_{\bF_{q^k}}|-|B'|}\bigl(1-q^{-kr\deg(P)}+q^{-kr\deg(P)}L(q^{k\deg(P)},m,r)\bigr)\Bigr)}
    {\prod_{|P|\in|Y_{\bF_{q^k}}|}\bigl(1-q^{-kr\deg(P)}+q^{-kr\deg(P)}L(q^{k\deg(P)},m,r)\bigr)} \\
  \shoveleft
    {= \Bigl(\prod_{|P|\in|B'|-|Y_{\bF_{q^k}}|} q^{-kr\deg(P)(n+1)}\Bigr)} \\
  \Bigl(\prod_{|P|\in|B'|\cap|Y_{\bF_{q^k}}|} \frac{q^{-kr\deg(P)(n+1)}}{1-q^{-kr\deg(P)}+q^{-kr\deg(P)}L(q^{k\deg(P)},m,r)}\Bigr).
\end{multline*}
So \cref{eq.equidist} is satisfied, and thus $(U_{\ul{d}},\ev_{\ul{d}})$ equidistributes on $A/B$.
\end{proof}

\subsection{Application: Zeta functions and \texorpdfstring{$L$}{L}-functions of complete intersections}
Let $\kappa$ be a finite field of order $q$ and fix an algebraic closure $\kappabar$. Let $Y \subseteq \mathbb{P}^n_{\kappa}$ be a smooth quasi-projective subscheme of dimension $m+r$. With notation as in \cref{prop.tuples-equidistribution}, for $F \in U_{\ul{d}}$ we write $C_{\ul{F}}$ for the scheme-theoretic intersection $Y \cap V(F_1) \cap \ldots \cap V(F_r)$, a smooth quasi-projective subscheme of $\mathbb{P}^n_{\kappa(\ul{F})}$, where $\kappa(\ul{F})$ is the subfield of $\kappabar$ generated by the coefficients of $F$. 

Let $Z_{\ul{d}}$ be the random variable on $U_{\ul{d}}$ sending $\ul{F}$ to
\[ Z_{C_{\ul{F}}}(t)=\prod_{y \in |C_{\ul{F}}|} \frac{1}{1-t^{\deg{y}}}. \]
The following (combined with \cite[Theorem 2.2.1]{Howe.TheNegativeSigmaMomentGeneratingFunction}) generalizes \cite[Theorem 8.3.1]{Howe.RandomMatrixStatisticsAndZeroesOfLFunctionsViaProbabilityInLambdaRings}, which is the case of $r=1$. 

\begin{theorem}\label{theorem.ci-geometric-rv}
With notation as above, as $\ul{d}$ goes to $\infty$ in $I_{m+r,r}$ (see \cref{sss.tuples-ordering}), the $\Lambda$-distribution of $Z_{\ul{d}}$ converges to a binomial $\Lambda$-distribution with parameters
\[ p=\frac{[q]^{-r}L([q],m+r,r)}{1- [q]^{-r} + [q]^{-r}L([q],m+r,r)} \textrm { and } N=[Y(\overline{\kappa})], \]
i.e.,
\[ \lim_{d \in I_{m+r,r}} \mathbb{E}[\Exp_{\sigma}(Z_{\ul{d}}h_1)] = (1+p(h_1 + h_2 + \ldots))^N. \]
\end{theorem}
\begin{proof}
We will use the notation of \cref{prop.tuples-equidistribution}. For $P \in Y(\kappabar)$, we consider the random variable $\mathcal{X}_P$ on $A_P$ that sends a germ $f$ to 
\[ \begin{cases}\frac{1}{1-t} & \textrm{ if $f_i(P)=0$ \textrm{ for all $i$ }} \\
0 & \textrm{ otherwise. }\end{cases} \]
For $P \in \mathbb{P}^m(\kappabar)\backslash Y(\kappabar)$, we set $\mathcal{X}_P$ to be the trivial random variable. 
Then 
\[ Z_d= \int_{U_d \times \mathbb{P}^n_{\kappabar}/ U_d} \ev_d^{*} \mathcal{X}. \]
By \cref{prop.tuples-equidistribution}, $(U_{\ul{d}}, \ev_{\ul{d}})$ equidistributes on $A/\mathbb{P}^m(\kappabar)$, thus, by \cref{theorem.abstract-independence}, the asymptotic $\sigma$-moment generating function is
\[ \prod_{\mathbb{P}^m(\kappabar)} \mathbb{E}_{A/\mathbb{P}^m(\kappabar)}[\Exp_{\sigma}(\mathcal{X}h_1)] = \prod_{Y(\kappabar)} \mathbb{E}_{A_{Y(\kappabar)}/Y(\kappabar)}[\Exp_{\sigma}(\mathcal{X}h_1)] \]
where the equality is because the moment generating function of the trivial random variable over $\mathbb{P}^m(\kappabar)-Y(\kappabar)$ is identically $1$. 

For $P \in Y(\kappabar)$ of degree $k$, $\mathcal{X}_P$ is a Bernoulli random variable equal to $\frac{1}{1-t}$ (the unit in $W(\mathbb{C})$) with probability 
\[ \frac{[q^k]^{-r}L([q^k],m+r,r)}{1- [q^k]^{-r} + [q^k]^{-r}L([q^k],m+r,r)}. \]
Indeed, this can be checked on each ghost component using \cref{lemma.restriction-map}; the numerator gives the probability that an $r$-tuple of germs vanishing at a point are transverse at that point, while the denominator gives the probability that an $r$-tuple of germs either do not all vanish at a point or all vanish and are transverse.

It follows that $\Exp_{\sigma}(\mathcal{X}h_1)|_{Y(\kappabar)}$ is the pullback of
$(1+ p(h_1 + h_2 + \ldots))$ from $
\mbf{1}$ to $Y(\kappabar)$. Thus, using \cref{example.constant-product-is-power} for the second equality, 
\[ \prod_{Y(\kappabar)} \mathbb{E}_{A_{Y(\kappabar)}/Y(\kappabar)}[\Exp_{\sigma}(\mathcal{X}h_1)] = \prod_{Y(\kappabar)} (1+ p(h_1 + h_2 + \ldots)) = (1+ p(h_1 + h_2 + \ldots))^{[Y(\kappabar)]}.   \] 
\end{proof}

\subsubsection{}

We now prove \cref{maintheorem.complete-intersections}. We continue with the notation above, except we now take $Y$ to be a smooth \emph{closed and geometrically connected} subscheme in order to agree with the setup in \cref{ss.intro-applications} (these conditions show up in the definition of vanishing cohomology and in the analysis of the top degree cohomology when establishing congruences modulo $W(\mathbb{C})^\bdd$). 

\begin{proof}[Proof of \cref{maintheorem.complete-intersections}]
We write $X_{\ul{d}}$ for the random variable on $U_{\ul{d}}$ as in \cref{ss.intro-applications} sending $\ul{F}$ to $\mathcal{L}_{C_{\ul{F}}}(t)$. For $Z_{\ul{d}}$ as above, we have, as in the $r=1$ case of \cite[\S8.4]{Howe.RandomMatrixStatisticsAndZeroesOfLFunctionsViaProbabilityInLambdaRings}, 
\begin{equation}\label{eq.vanishing-cohom-rv} X_{\ul{d}}=[q^{-m/2}]\left ((-1)^{m} Z_{\ul{d}} - [H^m(Y)] - (-1)^m\sum_{i=0}^{m-1} (-1)^i (1+[q^{m-i}]) [H^i(Y)]\right). \end{equation}

Because any constant random variables is independent to any other random variable, we find
\[ \mathbb{E}[\Exp_{\sigma}(X_{\ul{d}}h_1)] = \mathbb{E}[\Exp_{\sigma}((-1)^m [q^{-m/2}]Z_{\ul{d}})]\Exp_{\sigma}(\mu h_1) \]
for 
\[ \mu= - [q^{-m/2}][H^m(Y)] - (-1)^m\sum_{i=0}^{m-1} (-1)^i ([q^{-m/2}]+[q^{m/2-i}]) [H^i(Y)]. \]
To obtain \cref{eq.ci-main-theorem-limit}, it remains to note that, by \cref{theorem.ci-geometric-rv},
\[
\Exp_{\sigma}([q^{-m/2}]Z_{\ul{d}})]=(1+ p([q^{-m/2}]h_1 + [q^{-m}]h_2 + \ldots))^{[Y(\kappabar)]} 
\]
and thus also, by \cite[Theorem 2.2.1]{Howe.TheNegativeSigmaMomentGeneratingFunction}, 
\[
\Exp_{\sigma}(-[q^{-m/2}]Z_{\ul{d}})]=(1+ p(-[q^{-m/2}]e_1 + [q^{-m}]e_2 - \ldots))^{[Y(\kappabar)]}. 
\]

It remains just to establish the claimed comparisons mod $[q^{-1/2}]W(\mathbb{C})^\bdd$. This is nearly identical to the proof of \cite[Proposition 9.2.2]{Howe.RandomMatrixStatisticsAndZeroesOfLFunctionsViaProbabilityInLambdaRings} after we establish 
\[ p \equiv [q^{-r}] \mod [q^{-(m+1+r)}] W(\mathbb{C})^\bdd. \]
But this is immediate if we note $L([q],m+r,r)=\prod_{j=0}^{r-1}(1-[q]^{-(m+r-j)})$ then expand
\[ p= \frac{[q]^{-r}L([q],m+r,r)}{1- [q]^{-r} + [q]^{-r}L([q],m+r,r)}=\frac{[q^{-r}]-[q^{-(m+1)-r]}+\ldots}{1-[q^{-(m+1)-r}]+\ldots}. \]
\end{proof}

\section{Semiample equidistribution}\label{s.semiample}

In this section, we first establish an equidistribution result for sections of semiample bundles using the generalization of Poonen's sieve in \cite{ErmanWood.SemiampleBertiniTheoremsOverFiniteFields}, \cref{prop.semiample-equidistribution}. We then combine \cref{prop.semiample-equidistribution} with \cref{theorem.abstract-independence} to compute,  in \cref{theorem.hirzebruch-2d}, the asymptotic $\Lambda$-distribution of the zeta functions of curves of bidegree $(2,d)$ on Hirzebruch surfaces (generalizing the computation of the classical distribution of rational points given in \cite[Theorem 9.9-(b)]{ErmanWood.SemiampleBertiniTheoremsOverFiniteFields}). 

\subsection{Establishing equidistribution}

Let $Y$ be a smooth, projective scheme (integral but not necessarily geometrically integral) of dimension $m$ over $\bF_q$ with $q$ a power of a prime $p$. Consider a very ample divisor $D$ on $Y$ and a globally generated divisor $E$ on $Y$.
Let $\pi$ be the map given by the complete linear series on $E$:
  \[
    \pi:Y\xrightarrow{|E|}\bP^M_{\bF_q}.
  \]

Define $R_{n,d}:=H^0(Y,\sheaf{O}_Y(nD+dE))$ and for $F\in R_{n,d}$, write $H_F$ for the corresponding divisor in $|nD+dE|$.

Suppose $z \in |\pi(Y)| \subseteq |\mathbb{P}^M_{\mathbb{F}_q}|$ is a closed point and $z^{(1)}:=\Spec(\sheaf{O}_{\bP^M_{\bF_q},,z}/\frakm_z^2)$ the first-order infinitesimal neighborhood of $z$.
For $y \in |Y|$ a closed point in $\pi^{-1}(z)$, let $y^{(1)}=\Spec(\sheaf{O}_{Y,y}/\frakm_y^2)$. For any finite subscheme $W\subset \mathbb{P}^M_{\mathbb{F}_q}$, define $Y_W=Y\times_{\mathbb{P}^M_{\mathbb{F}_q}} W$.

Given a section $F\in R_{n,d}$, $H_F$ is smooth at a closed point $y \in\pi^{-1}(z)$ if and only if $F$ does not vanish under the restriction map
  \[
    R_{n,d}
    \to H^0(y^{(1)},\sheaf{O}_{y^{(1)}}(nD))
    \cong \sheaf{O}_{Y,y}/\frakm_y^2.
  \]
where the latter isomorphism depends on a choice of trivialization of $\sheaf{O}_{y^{(1)}}(nD)$, but the condition of non-vanishing does not. 
This restriction map factors as
  \[
    R_{n,d}
    \xrightarrow{\alpha} H^0(Y_{y^{(1)}},\sheaf{O}_{Y_{y^{(1)}}}(nD))
    \rightarrow \sheaf{O}_{Y,y}/\frakm_y^2.
  \]
Set $\sheaf{F}=\pi_*(\sheaf{O}_Y(nD))$, so $\sheaf{F}(d)\cong \pi_*(\sheaf{O}_Y(nD+dE))$.

There is a natural map
\begin{equation*}
  \sheaf{F}(d) \otimes_{\sheaf{O}_{\bP^M}}\sheaf{O}_{z^{(1)}}
  \xrightarrow{\beta} H^0\bigl(Y_{y^{(1)}},\sheaf{O}_{Y_{y^{(1)}}}(nD)\bigr).
\end{equation*}
Note that $\alpha$ is the composition of $\beta$ and the natural restriction map
\begin{equation}\label{eq.restriction-map-semiample}R_{n,d} \rightarrow \sheaf{F}(d) \otimes_{\sheaf{O}_{\bP^M}}\sheaf{O}_{z^{(1)}}.\end{equation}
As explained in the proof of \cite[Lemma 5.2(a)]{ErmanWood.SemiampleBertiniTheoremsOverFiniteFields}, Serre vanishing implies that for $d\gg 0$ this restriction map \cref{eq.restriction-map-semiample} is surjective.

Let $\phi:\bP^M_{\Fqbar}\to\bP^M_{\bF_q}$ (resp. $\phi_k:\bP^M_{\Fqbar}\to\bP^M_{\bF_{q^k}}$) be the natural map. Let $\pi':Y_{\Fqbar}\to\bP^M_{\Fqbar}$  be the base change of $\pi$ relative to $\phi$. We write the homogeneous coordinates on $\mathbb{P}^{M}_{\fq}$ as $x_0, \ldots, x_M$, and for each $P \in (\pi(Y))(\fqbar)$, we fix a $0 \leq j_P \leq M$ such that $x_{j_P}$ does not vanish at $P$; we make this choice so that $j_P$ is constant on orbits. For $F \in R_{n,d}(\Fqbar)$, we write $F_P$ for the image of $F/x_{j_P}^d$ in $\phi^*\sheaf{F} \otimes_{\sheaf{O}_{\mathbb{P}^M_{\fqbar}}}\sheaf{O}_{P^{(1)}}$.

\begin{proposition}\label{prop.semiample-equidistribution}
With notation as above, set $B=(\pi(Y))(\Fqbar)$. For each $P\in B$, let $A_P$ be the set of $g_P\in \phi^*\sheaf{F} \otimes_{\sheaf{O}_{\mathbb{P}^M_{\fqbar}}} \sheaf{O}_{P^{(1)}}$ such that the image of $g_P$ in $\sheaf{O}_{Q^{(1)}}$ is nonzero for all $Q\in\pi'^{-1}(P)$.

Set $A=\bigsqcup_{P\in B}A_P$.
Let $U_d$ be the set of $F\in R_{n,d}(\Fqbar)$ such that for all $P\in B$, the image of $F$ in $\sheaf{O}_{Q^{(1)}}$ is nonzero for all $Q \in\pi'^{-1}(P)$.
Each of $A$, $B$, and $U_d$ are admissible $\bZ$-sets with the geometric Frobenius action.

Let $A\to B$ be the map $g_P\mapsto P$ and $\ev_d:U_d\times B\to A$ the map $(F,P)\mapsto F_P$. If $n\geq \max\{(\dim \pi(Y))(m+1)-1,(\dim \pi(Y))p+1\}$, then $(U_d,\ev_d)$ equidistributes on $A/B$ in the sense of \cref{def.equidistributes}.
\end{proposition}

\begin{proof}
For an admissible $\bZ$-subset $B'\subseteq B_{\bfk}$ of finite cardinality, we have
\begin{equation*}
  \Hom_{B_{\bfk}}(B',A_{\bfk})
  = \prod_{|P| \in |B'|} \Hom_{B_{\bfk}}(|P|, A_{\bf{k}}).
\end{equation*}
If we identify the orbit $|P| \subseteq B' \subseteq \mathbb{P}^M(\fqbar)$ with a closed point of $\mathbb{P}^M_{\mathbb{F}_{q^k}}$, then we obtain a canonical identification of 
$\Hom_{B_{\bfk}}(|P|, A_{\bf{k}})$ with a subset $A_{|P|}$ of  $\sheaf{F}\otimes_{\sheaf{O}_{\bP^M_{\mathbb{F}_{q^k}}}}\sheaf{O}_{|P|^{(1)}}$.
Viewing $U_{d,\bfk}(\bfone)$ as the set of $F\in R_{n,d}(\bF_{q^k})$ such that the image $F_{|P|}$ of $F$ in $\sheaf{F}\otimes_{\sheaf{O}_{\bP^M_{\mathbb{F}_{q^k}}}}\sheaf{O}_{|P|^{(1)}}$ lies in $A_{|P|}$ for all $|P|\in|\bP^n_{\bF_{q^k}}|$, the map $\ev_{d,B'}$ sends $F$ to the tuple $(F_{|P|})_{|P|\in|B'|}$.

To establish \cref{eq.equidist}, it suffices to show equality for each singleton $\{(F_{|P|})_{|P| \in |B'|}\}$.
On the right side we have
\begin{equation*}
  \mu_{\Hom_{B_{\bfk}}(B',A_{\bfk})}(\{(F_{|P|})_{|P|}\})
  = \prod_{|P|\in|B'|} \frac{1}{\# A_{|P|}}.
\end{equation*}
The left side can be viewed as a conditional probability:
\begin{align*}
  \MoveEqLeft \bigl((\ev_{d,B'})_*\mu_{U_{d,\bfk}(\bfone)}\bigr)(\{(F_{|P|})_{|P|}\}) \\
  &= \mu_{U_{d,\bfk}(\bfone)}(\ev_{d,B'}^{-1}(\{(F_{|P|})_{|P|}\}) \\
  &= \frac{\#\{G\in R_{n,d}(\bF_{q^k})\mid G_{|P|}=F_{|P|} \text{ for all $|P|\in |B'|$ and $G\in U_{d,\bfk}(\bfone)$}\}/\# R_{n,d}(\bF_{q^k})}{\# U_{d,\bfk}(\bfone)/\# R_{n,d}(\bF_{q^k})}.
\end{align*}

Since the restriction maps of \cref{eq.restriction-map-semiample} (or rather their analogs over $\mathbb{F}_{q^k}$) are surjective for $d\gg 0$, we can phrase local probabilities at $|P|$ in terms of $\phi_k^*\sheaf{F}(d) \otimes_{\sheaf{O}_{\mathbb{P}^M_{\mathbb{F}_{q^k}}}} \sheaf{O}_{|P|^{(1)}}$ instead of $R_{n,d}$. Thus, by \cite[Theorem 3.1]{ErmanWood.SemiampleBertiniTheoremsOverFiniteFields}, as $d\to\infty$, this converges to
\begin{multline*}
  \frac
    {
      \Bigl(
        \prod_{|P|\in|B'|}
        \frac{1}{\# \phi_k^*\sheaf{F} \otimes_{\sheaf{O}_{\bP^M_{\mathbb{F}_{q^k}}}}\sheaf{O}_{|P|^{(1)}} }
      \Bigr)
      \Bigl(
        \prod_{|P|\in|\pi(Y)_{\bF_{q^k}}|-|B'|}
        \frac{\# A_{|P|}}{\# \phi_k^*\sheaf{F} \otimes_{\sheaf{O}_{\bP^M_{\mathbb{F}_{q^k}}}}\sheaf{O}_{|P|^{(1)}} }
      \Bigr)
    }
    {
      \prod_{|P|\in|\pi(Y)_{\bF_{q^k}}|}
      \frac{\# A_{|P|}}{\# \phi_k^*\sheaf{F} \otimes_{\sheaf{O}_{\bP^M_{\mathbb{F}_{q^k}}}}\sheaf{O}_{|P|^{(1)}}}
    }\\
  = \prod_{|P|\in|B'|}\frac{1}{\# A_{|P|}}.
\end{multline*}
So \cref{eq.equidist} is satisfied, and thus $(U_d,\ev_d)$ equidistributes on $A/B$.
\end{proof}

\subsection{Application: Zeta functions of curves on Hirzebruch surfaces}\label{ss.hirzebruch}

In \cite[Theorem 9.9]{ErmanWood.SemiampleBertiniTheoremsOverFiniteFields}, the semiample Bertini theorem is used to compute the asymptotic distribution of some point-counting random variables for smooth curves on Hirzebruch surfaces. Using our methods, this can be extended to compute the full $\Lambda$-distributions  --- the classical distributions in \cite[Theorem 9.9]{ErmanWood.SemiampleBertiniTheoremsOverFiniteFields} are equivalent to the restriction of the $\Lambda$-distributions to $\mathbb{Z}[h_1] \subseteq \Lambda$. We illustrate this below in the case of bidegree $(2,d)$ curves on Hirzebruch surfaces (\cite[Theorem 9.9-(b)]{ErmanWood.SemiampleBertiniTheoremsOverFiniteFields}).

\subsubsection{}Let $\kappa$ be a finite field of order $q$, let $Y = \mathrm{Proj}_{\mathbb{P}^1_{\kappa}} (\Sym^\bullet (\sheaf{O}\oplus\sheaf{O}(a)))$ for $a \geq 0$ be a Hirzebruch surface over $\kappa$, and write $\pi:Y \rightarrow \mathbb{P}^1_{\kappa}$ for its natural projection.
We let $E$ be the divisor on $Y$ of the fiber over $\infty$ on $\mathbb{P}^1_\kappa$ (so that $\pi$ is induced by $E$) and let $D$ the class of a hyperplane section in the relative proj construction. We write $\sheaf{O}(i,j)=\sheaf{O}(iD + jE)$, a line bundle on $Y$. 

Let $U_d$ be the admissible $\mathbb{Z}$-set of global sections of $\sheaf{O}(2,d)$ on $Y_{\kappabar}$ with smooth vanishing locus. For $F$ in $U_d$, we write $\kappa(F)$ for the subfield of $\kappabar$ generated by the coefficients of $F$ and $\kappa$, and we view the vanishing locus $V(F)$ as a scheme over $\kappa(F)$. 

Let $Z_d$ be the random variable on $U_d$ sending $F$ to the zeta function 
\[ Z_{V(F)}(t) = \prod_{y \in |V(F)|} \frac{1}{1-t^{\deg{y}}} = \prod_{z \in |\mathbb{P}^1_{\kappa(F)}|} Z_{\pi^{-1}(z) \cap V(F)}(t). \]  

\begin{theorem}\label{theorem.hirzebruch-2d}
With notation as above,
\begin{multline*}\lim_{d \rightarrow \infty} \mathbb{E}[\Exp_{\sigma}(Z_d h_1)]=\\
\left( \frac{([q]^{2}-1)([q]-1)\left(\sum_{j \geq 0} h_j\right) + \frac{[q]^{4}-[q]^{2}}{2}\left(\sum_{j\geq 0}h_j\right)^2  + \frac{([q]^{2}-[q])^2}{2} \left(\sum_{j \geq 0}h_j(\ul{t}^2)\right)}{{[q]^{4}-[q]^{2}-[q]+1}} \right)^{\mbb{P}^1(\kappabar)}.\end{multline*}
\end{theorem}
\begin{proof}
We adopt the notation of \cref{prop.semiample-equidistribution} for our choice of $Y$, $D$, and $E$ above. We note that, although $n=2$ does not satisfy the bounds given in \cref{prop.semiample-equidistribution}, by \cite[Proposition 8.2]{ErmanWood.SemiampleBertiniTheoremsOverFiniteFields}, the application of \cite[Theorem 3.1]{ErmanWood.SemiampleBertiniTheoremsOverFiniteFields} in the proof of \cref{prop.semiample-equidistribution} is still valid in this specific setting, so that we have equidistribution.  

For $P \in \mathbb{P}^1(\kappabar)$, let $\mathcal{X}_P$ be the random variable on $A_P$ that sends a germ $g_P$ to $Z_{V(\overline{g}_P)}(t)$, where here the vanishing locus is taken inside of $Y_P\cong \mathbb{P}^1_{\fqbar}$, $\overline{g}_P$ is the induced element of $H^0(Y_P, \sheaf{O}(2D)) \cong H^0(\mathbb{P}^1_{\fqbar}, \sheaf{O}(2))$, and we treat $V(\overline{g}_P)$ as being defined over $\kappa(g_P)$ to obtain a zeta function.   

We then have 
\[ X_d = \int_{U_d \times \mathbb{P}^1(\kappabar) / U_d} \ev^* \mc{X} \]
and thus, by \cref{theorem.abstract-independence}, 
\[ \lim_{d \rightarrow \infty} \mathbb{E}[\Exp_{\sigma}(Z_d h_1)]= \prod_{\mathbb{P}^1(\kappabar)} \mathbb{E}_{A/\mathbb{P}^1(\kappabar)}[\Exp_{\sigma}(\mathcal{X}h_1)]. \]

We now compute the $\sigma$-moment generating function for $\mc{X}_P$, $P \in \mathbb{P}^1(\kappabar)$, using \cref{lemma.restriction-map}. Let $q_P=\#\kappa(P)$.
We note that $\res_k(\mc{X}_P)$ can be viewed as a function on the germs at $P$ defined over $\kappa(P)_k$, the degree $k$ extension of $\kappa(P)$ in $\kappabar$. On such a germ $g$, it takes value
\begin{enumerate}
\item $\frac{1}{1-t}$ if $\overline{g}$ is the square of a single factor. The number of such cases is 
\[  (q_P^{2k}-1) \cdot (q_P^{3k}-q_P^{2k}).\]
Here the $q_P^{2k}-1=(q_P^k+1)(q_P^k-1)$ is the number of points in $\mathbb{P}^1(\kappa(P)_k)$ times the number of degree two homogeneous equations vanishing at such a point with multiplicity two, and  the factor $q_P^{3k}-q_P^{2k}$ is the number of possible smooth extensions $g$ of each $\overline{g}$ (which correspond to degree two polynomials that don't have a zero at the same point --- cf. \cite[Lemma 9.8 and preceding paragraph]{ErmanWood.SemiampleBertiniTheoremsOverFiniteFields})\footnote{One can also compare this computation with \cite[proof of Proposition 9.9-(b)]{ErmanWood.SemiampleBertiniTheoremsOverFiniteFields}, but note that there is a typo in the corresponding computation in that proof: the second $(q-1)(q+1)$ appearing should in fact be our $q^3-q^2$.}.
\item $\left(\frac{1}{1-t}\right)^2$  if $\overline{g}$ splits into two distinct factors over $\kappa(P)_k$. The number of such cases is 
\[ \frac{(q_P^{k}+1)q_P^{k}}{2}(q_P^k -1)q_P^{3k}=\frac{q_P^{2k}-1}{2}q_P^{4k}\]
where $\frac{(q_P^{k}+1)q_P^{k}}{2}(q_P^k -1)$ is the number of pairs of distinct points in $\mathbb{P}^1(\kappa(P)_k)$ times the number of degree two homogeneous equations vanishing exactly at such a pair, and $q_P^{3k}$ counts the number of smooth extensions $g$ of each $\overline{g}$ (which correspond to arbitrary degree $3$ polynomials). 
\item  $\frac{1}{1-t^2}$ if $\overline{g}$ is irreducible over $\kappa(P)_k$. The number of such cases is 
\[ \frac{q_P^{2k}-q_P^k}{2}(q_P^k-1)q_P^{3k}= \frac{(q_P^k-1)^2}{2}q_P^{4k} \] 
where $\frac{q_P^{2k}-q_P^k}{2}(q_P^k-1)$ is the number of degree $2$ closed points in $\mathbb{P}^1_{\kappa(P)_k}$ times the number of degree two homogeneous equations vanishing exactly at such a point, and $q_P^{3k}$ counts the number of smooth extensions $g$ of each $\overline{g}$ (which correspond to arbitrary degree $3$ polynomials). 
\end{enumerate}
Note that the total number of cases adds up to $q_P^{6k}-q_P^{4k}-q_P^{3k}+q_P^{2k}$, i.e. this is the denominator for the probability of each case occurring. 

Now we note that 
\[ \Exp_{\sigma}\left( \frac{1}{1-t} h_1 \right) = \sum_{j \geq 0} h_j \] 
and 
\[ \Exp_{\sigma}\left(\left(\frac{1}{1-t}\right)^2 h_1\right)=\left(\Exp_{\sigma}\left(\frac{1}{1-t} h_1\right)\right)^2=\left(\sum_{j \geq 0} h_j\right)^2. \]
The formula for $\Exp_{\sigma}(\frac{1}{1-t^2})$ is more complicated, but we only need the first component, which is straightforward:
\[ \Exp_{\sigma}\left(\frac{1}{1-t^2}h_1\right)_1 = \sum_{\tau} \left(h_\tau \circ \frac{1}{1-t^2}\right)_1 m_\tau =\sum_{\tau} m_{2\tau}=\sum_{\tau} m_\tau(\ul{t}^2)=\sum_j h_j(\ul{t}^2). \]
The first equality is \cite[Example 2.5.2]{Howe.RandomMatrixStatisticsAndZeroesOfLFunctionsViaProbabilityInLambdaRings} and the second follows because $\frac{1}{1-t^2}=[\mbf{2}]$ so that $h_j \circ \frac{1}{1-t^2}=[\Sym^j (\mbf{2})]$; indeed, $\Sym^j(\mbf{2})$ has one fixed point if $j$ is even and no fixed points otherwise. Combining our computation above of the probability of each value $\bullet$ with these computations of $\Exp_{\sigma}(\bullet)_1$, we obtain 
\begin{multline*} \mathbb{E}_1[ \Exp_{\sigma}(\res_k(\mc{X}_P h_1))] = \\\frac{(q_P^{2k}-1)(q_P^{k}-1)\left(\sum_{j \geq 0} h_j\right) + \frac{q_P^{4k}-q_P^{2k}}{2}\left(\sum_{j\geq 0}h_j\right)^2  + \frac{(q_P^{2k}-q_P^k)^2}{2} \left(\sum_{j \geq 0}h_j(\ul{t}^2)\right)}{{q_P^{4k}-q_P^{2k}-q_P^k+1}}.\end{multline*}
Applying \cref{lemma.restriction-map}, and comparing $k$th components, we find
\begin{multline*} \mathbb{E}[ \Exp_{\sigma}(\mc{X}_P h_1)] = \\\frac{([q_P]^{2}-1)([q_P]-1)\left(\sum_{j \geq 0} h_j\right) + \frac{[q_P]^{4}-[q_P]^{2}}{2}\left(\sum_{j\geq 0}h_j\right)^2  + \frac{([q_P]^{2}-[q_P])^2}{2} \left(\sum_{j \geq 0}h_j(\ul{t}^2)\right)}{{[q_P]^{4}-[q_P]^{2}-[q_P]+1}}.\end{multline*}
The function on $\mathbb{P}^1(\kappabar)$ sending $P$ to this series is the pullback from $\bf1$ of 
\[ \frac{([q]^{2}-1)([q]-1)\left(\sum_{j \geq 0} h_j\right) + \frac{[q]^{4}-[q]^{2}}{2}\left(\sum_{j\geq 0}h_j\right)^2  + \frac{([q]^{2}-[q])^2}{2} \left(\sum_{j \geq 0}h_j(\ul{t}^2)\right)}{{[q]^{4}-[q]^{2}-[q]+1}}. \]
Thus, applying \cref{example.constant-product-is-power} to compute  $\prod_{\mathbb{P}^1(\kappabar)} \mathbb{E}_{A/\mathbb{P}^1(\kappabar)}[\Exp_{\sigma}(\mathcal{X}h_1)]$, we obtain the claimed expression. 
\end{proof}

\bibliographystyle{plain}
\bibliography{references, preprints}

\begin{thebibliography}{10}

\bibitem{Bertucci.TaylorConditionsOverFiniteFields}
Matthew Bertucci.
\newblock Taylor conditions over finite fields.
\newblock {\em ar{X}iv:2412.08744}, 2024.

\bibitem{Bilu.MotivicEulerProductsAndMotivicHeightZetaFunctions}
Margaret Bilu.
\newblock Motivic {E}uler products and motivic height zeta functions.
\newblock {\em Mem. Amer. Math. Soc.}, 282(1396):v+185, 2023.

\bibitem{BiluDasHowe.SpecialValuesOfMotivicEulerProducts}
Margaret Bilu, Ronno Das, and Sean Howe.
\newblock Special values of motivic euler products.
\newblock {\em In preparation}.

\bibitem{BiluDasEtAl.ZetaStatisticsAndHadamardFunctions}
Margaret Bilu, Ronno Das, and Sean Howe.
\newblock Zeta statistics and {H}adamard functions.
\newblock {\em Adv. Math.}, 407:Paper No. 108556, 68, 2022.

\bibitem{BiluHowe.MotivicRandomVariables}
Margaret Bilu and Sean Howe.
\newblock Motivic random variables.
\newblock {\em In preparation}.

\bibitem{BiluHowe.MotivicEulerProductsInMotivicStatistics}
Margaret Bilu and Sean Howe.
\newblock Motivic {E}uler products in motivic statistics.
\newblock {\em Algebra Number Theory}, 15(9):2195--2259, 2021.

\bibitem{BucurKedlaya.TheProbabilityThatACompleteIntersectionIsSmooth}
Alina Bucur and Kiran~S. Kedlaya.
\newblock The probability that a complete intersection is smooth.
\newblock {\em J. Th\'{e}or. Nombres Bordeaux}, 24(3):541--556, 2012.

\bibitem{ErmanWood.SemiampleBertiniTheoremsOverFiniteFields}
Daniel Erman and Melanie~Matchett Wood.
\newblock Semiample {B}ertini theorems over finite fields.
\newblock {\em Duke Math. J.}, 164(1):1--38, 2015.

\bibitem{Howe.RandomMatrixStatisticsAndZeroesOfLFunctionsViaProbabilityInLambdaRings}
Sean Howe.
\newblock Random matrix statistics and zeroes of {L}-functions via probability in $\lambda$-rings.
\newblock {\em ar{X}iv:2412.19295}, 2024.

\bibitem{Howe.TheNegativeSigmaMomentGeneratingFunction}
Sean Howe.
\newblock The negative $\sigma$-moment generating function.
\newblock {\em ar{X}iv:2505.01205}, 2025.

\bibitem{Poonen.BertiniTheoremsOverFiniteFields}
Bjorn Poonen.
\newblock Bertini theorems over finite fields.
\newblock {\em Ann. of Math. (2)}, 160(3):1099--1127, 2004.

\end{thebibliography}

\end{document}